\documentclass[11pt]{amsart}
\usepackage{amsmath}
\usepackage{amsthm}
\usepackage{tikz}
\usepackage{tikz-cd}
\usepackage{hyperref}
\usepackage[left=2.5cm,right=2.5cm,top=3cm,bottom=3cm]{geometry}
\usepackage{amssymb}
\usepackage{amscd}
\usepackage{latexsym}
\usepackage{epsfig}
\usepackage{graphicx}
\usepackage{psfrag}
\usepackage{caption}
\usepackage{rotating}
\usepackage{mathtools}
\usepackage{xypic}

\allowdisplaybreaks[3]

\usepackage{url}
\usepackage{cite}
\usepackage{array,}

\usepackage{blindtext}
\usepackage{scrextend}
\addtokomafont{labelinglabel}{\sffamily}

\newtheorem{theorem}{Theorem}
\newtheorem{proposition}[theorem]{Proposition}

\newtheorem{lemma}[theorem]{Lemma}

\theoremstyle{definition}

\newtheorem{remark}[theorem]{Remark}
\newtheorem*{question*}{Motivating Question}

\DeclareMathOperator{\GL}{GL}

\DeclareMathOperator{\Adj}{Ad}

\makeatletter
\DeclareRobustCommand{\element}[1]{\@element#1\@nil}
\def\@element#1#2\@nil{%
	#1%
	\if\relax#2\relax\else\MakeLowercase{#2}\fi}
\makeatother

\title[Kostant--Toda lattices and the universal centralizer]{Kostant--Toda lattices and the universal centralizer}

\author[Peter Crooks]{Peter Crooks}
\address{Department of Mathematics, Northeastern University, 360 Huntington Ave., Boston, MA 02115, USA}
\email{~~~peter.d.crooks@gmail.com}


\begin{document}
	
	\maketitle

\vspace{10pt}

\begin{abstract}
To each complex semisimple Lie algebra $\mathfrak{g}$ decorated with appropriate data, one may associate two completely integrable systems. One is the well-studied Kostant--Toda lattice, while the second is an integrable system defined on the universal centralizer $\mathcal{Z}_{\mathfrak{g}}$ of $\mathfrak{g}$. These systems are similar in that each exploits and closely reflects the invariant theory of $\mathfrak{g}$, as developed by Chevalley, Kostant, and others. One also has Kostant's description of level sets in the Kostant--Toda lattice, which turns out to suggest deeper similarities between the two integrable systems in question. 

We study relationships between the two aforementioned integrable systems, partly to understand and contextualize the similarities mentioned above. Our main result is a canonical open embedding of a flow-invariant open dense subset of the Kostant--Toda lattice into $\mathcal{Z}_{\mathfrak{g}}$. Secondary results include some qualitative features of the integrable system on $\mathcal{Z}_{\mathfrak{g}}$.                         
\end{abstract}

\vspace{15pt}

\tableofcontents

\section{Introduction}

\subsection{Motivation and context}
The finite, non-periodic Toda lattice is a cornerstone of classical integrable systems theory \cite{Toda1,Toda2,Moser}, and it makes well-documented connections to representation theory \cite{AdlervanMoerbeke,Kim,KostantFlag,Givental,Rietsch}. These connections are perhaps best elucidated through Kostant's Lie-theoretic realization of the Toda lattice \cite{KostantSolution}, a completely integrable system $\overline{F}:\mathcal{O}_{\text{KT}}\rightarrow\mathbb{C}^r$ that one sometimes calls the \textit{Kostant--Toda lattice}. Kostant defines this system for each rank-$r$ complex semisimple Lie algebra $\mathfrak{g}$ decorated with certain Lie-theoretic data, and he gives a particularly elegant description of this system's level sets. His description involves the adjoint group $G$ of $\mathfrak{g}$, and is roughly stated as follows: each level set $\overline{F}^{-1}(z)$ admits an explicit open embedding into the $G$-stabilizer of a suitable regular element in $\mathfrak{g}$.

The preceding discussion features prominetly in \cite{AbeCrooks2}, where (among other things) Abe and the current author speculate about a possible role to be played by the \textit{universal centralizer} $\mathcal{Z}_{\mathfrak{g}}$ of $\mathfrak{g}$. This smooth symplectic variety has received some attention in representation-theoretic contexts \cite{BalibanuUniversal,GinzburgKostant}, and it turns out to come equipped with a completely integrable system $\widetilde{F}:\mathcal{Z}_{\mathfrak{g}}\rightarrow\mathbb{C}^r$. Each level set of the aforementioned system is canonically isomorphic to the $G$-stabilizer of an explicit regular element in $\mathfrak{g}$. Comparing this fact with Kostant's description of the level sets $\overline{F}^{-1}(z)$, we are prompted to ask the following question: to what extent are the integrable systems $\overline{F}:\mathcal{O}_{\text{KT}}\rightarrow\mathbb{C}^r$ and $\widetilde{F}:\mathcal{Z}_{\mathfrak{g}}\rightarrow\mathbb{C}^r$ related?

\subsection{Outline of results}\label{Subsection: Outline of results}
Our main contribution is an answer to the question posed above. We also establish some incidental facts about the integrable system $\widetilde{F}:\mathcal{Z}_{\mathfrak{g}}\rightarrow\mathbb{C}^r$, facts that we believe to be independently interesting and potentially useful in other contexts.  

To outline our work, we let $\mathfrak{g}$ and $G$ be exactly as described earlier and we fix the following data:
\begin{itemize}
\item[(i)] homogeneous, algebraically independent generators $f_1,\ldots,f_r$ of the invariant ring $\mathrm{Sym}(\mathfrak{g}^*)^G$;
\item[(ii)] a Cartan subalgebra $\mathfrak{t}\subseteq\mathfrak{g}$;
\item[(iii)] a choice of simple roots $\Delta\subseteq\mathfrak{t}^*$;
\item[(iv)] for each $\alpha\in\Delta$, choices of $e_{\alpha}\in\mathfrak{g}_{\alpha}$ and $e_{-\alpha}\in\mathfrak{g}_{-\alpha}$ that pair to $1$ under the Killing form, where $\mathfrak{g}_{\pm\alpha}\subseteq\mathfrak{g}$ is the root space associated to $\pm\alpha$;
\item[(v)] an $\mathfrak{sl}_2$-triple $(\xi,h,\eta)\in\mathfrak{g}^{\oplus 3}$, where $h\in\mathfrak{t}$ is defined by $\alpha(h)=-2$ for all $\alpha\in\Delta$, $\xi:=\sum_{\alpha\in\Delta}e_{-\alpha}$, and $\eta:=\sum_{\alpha\in\Delta}c_{\alpha}e_{\alpha}$ for suitable coefficients $c_{\alpha}\in\mathbb{C}^{\times}$.
\end{itemize} 
This $\mathfrak{sl}_2$-triple canonically determines a \textit{Slodowy slice} (a.k.a \textit{Kostant section}) $\mathcal{S}\subseteq\mathfrak{g}$, and the \textit{universal centralizer} is then the following smooth, $2r$-dimensional, closed subvariety of $G\times\mathcal{S}$:
\begin{equation}\label{Equation: Universal centralizer}\mathcal{Z}_{\mathfrak{g}}:=\{(g,x)\in G\times\mathcal{S}: g\in G_x\},\end{equation} where $G_x$ denotes the $G$-stabilizer of $x\in\mathcal{S}$ under the adjoint action. The universal centralizer has a canonical symplectic structure (Proposition \ref{Proposition: Symplectic subvariety}), in which context we consider the functions $\widetilde{f}_1,\ldots,\widetilde{f}_r:\mathcal{Z}_{\mathfrak{g}}\rightarrow\mathbb{C}$ defined by
$$\widetilde{f}_i(g,x)=f_i(x),\quad (g,x)\in\mathcal{Z}_{\mathfrak{g}}$$
for all $i\in\{1,\ldots,r\}$. In particular, we note that $\widetilde{F}:=(\widetilde{f}_1,\ldots,\widetilde{f}_r):\mathcal{Z}_{\mathfrak{g}}\rightarrow\mathbb{C}^r$ is a completely integrable system (Proposition \ref{Proposition: Integrable system}) with level sets satisfying $\widetilde{F}^{-1}(\widetilde{F}(g,x))\cong G_x$ for all $(g,x)\in\mathcal{Z}_{\mathfrak{g}}$ (Proposition \ref{Proposition: Moment map image}). This system turns out to be quite concrete, an idea that we emphasize through an explicit construction of \textit{Carath\'eodory--Jacobi--Lie coordinates} (see the end of Section \ref{Subsection: The integrable system}).   

Now consider the $2r$-dimensional locally closed subvariety
\begin{equation}\label{Equation: Kostant--Toda definition}\mathcal{O}_{\text{KT}}:=\xi+\mathfrak{t}+\sum_{\alpha\in\Delta}(\mathfrak{g}_{\alpha}\setminus\{0\}):=\bigg\{\xi+x+\sum_{\alpha\in\Delta}y_{\alpha}:x\in\mathfrak{t}\text{ and }y_{\alpha}\in\mathfrak{g}_{\alpha}\setminus\{0\}\text{ for all }\alpha\in\Delta\bigg\}\end{equation} of $\mathfrak{g}$, which is known to carry a distinguished symplectic form. Let $\overline{f}_1,\ldots,\overline{f}_r:\mathcal{O}_{\text{KT}}\rightarrow\mathbb{C}$ denote the restrictions of $f_1,\ldots,f_r$ to $\mathcal{O}_{\text{KT}}$, respectively. Kostant \cite{KostantSolution} proves $\overline{F}:=(\overline{f}_1,\ldots,\overline{f}_r):\mathcal{O}_{\text{KT}}\rightarrow\mathbb{C}^r$ defines a completely integrable system, hereafter called the Kostant--Toda lattice.

Our main result is the following relationship between the integrable systems $\overline{F}:\mathcal{O}_{\text{KT}}\rightarrow\mathbb{C}^r$ and $\widetilde{F}:\mathcal{Z}_{\mathfrak{g}}\rightarrow\mathbb{C}^r$.

\begin{theorem}\label{Theorem: Main theorem}
There exist a canonical open dense subset $\mathcal{V}\subseteq\mathcal{O}_{\emph{KT}}$ and a canonical open embedding of complex manifolds $\varphi:\mathcal{V}\rightarrow\mathcal{Z}_{\mathfrak{g}}$ satisfying the following conditions.
\begin{itemize}
\item[(i)] The open set $\mathcal{V}$ is a union of level sets of $\overline{F}$.  
\item[(ii)] We have a commutative diagram \begin{align*}
\xymatrixrowsep{4pc}\xymatrixcolsep{4pc}\xymatrix{
	\mathcal{V} \ar[rd]_{\overline{F}\big\vert_{\mathcal{V}}} \ar[rr]^{\varphi} & & \mathcal{Z}_{\mathfrak{g}} \ar[ld]^{\widetilde{F}} \\
	& \mathbb{C}^r & 
}.
\end{align*} 
\item[(iii)] For all $i\in\{1,\ldots,r\}$, $\varphi$ identifies the Hamiltonian vector field of $\overline{f}_i$ on $\mathcal{V}$ with the Hamiltonian vector field of $\widetilde{f}_i$ on the image of $\varphi$.
\end{itemize} 
\end{theorem}
 
\subsection{Organization}
Section \ref{Section: Notation and conventions} briefly describes some of the notation and conventions adopted throughout this paper. The main mathematical content then begins in Section \ref{Section: The universal centralizer}, which concerns the universal centralizer and its completely integrable system. Section \ref{Section: The Kostant--Toda lattice} subsequently studies aspects of the Kostant--Toda lattice, especially those aspects pertinent to Theorem \ref{Theorem: Main theorem}. A proof of Theorem \ref{Theorem: Main theorem} is then given in Section \ref{Section: The relationship}.  

\subsection*{Acknowledgements}
The author wishes to thank Ana B{\u{a}}libanu for enlightening conversations. He also gratefully acknowledges support from the Natural Sciences and Engineering Research Council of Canada [PDF--516638--2018].  

\section{Notation and conventions}\label{Section: Notation and conventions}
This paper works exclusively over $\mathbb{C}$, viewing it as the base field of anything whose definition presupposes a base field (e.g. a variety, manifold, Lie algebra, etc.). We deal extensively with affine varieties, using $\mathbb{C}[X]$ to denote the coordinate ring of any affine variety $X$. Furthermore, we always understand ``group action'' as meaning ``left group action''.

We make many statements concerning the openness, closedness, closure, and denseness of sets. Whenever any such statement is ambiguous about whether it uses the algebro-geometric / Zariski topology versus the Euclidean / complex-analytic topology, the latter topology is being used implicitly.               

Throughout this paper, $\mathfrak{g}$ is a finite-dimensional, semisimple, rank-$r$ Lie algebra with adjoint group $G$ and associated exponential map $\exp:\mathfrak{g}\rightarrow G$. The second paragraph of \ref{Subsection: Outline of results} chooses the additional Lie-theoretic data (i)--(v), and we now regard these as fixed for the duration of the paper. We also have an adjoint representation $\Adj:G\rightarrow\GL(\mathfrak{g})$, $g\mapsto \Adj_g$, through which $G$ acts on $\mathfrak{g}$. Denote by $G_x\subseteq G$ (resp. $Gx\subseteq \mathfrak{g}$) the $G$-stabilizer (resp. $G$-orbit) of any $x\in\mathfrak{g}$, noting that the Lie algebra of $G_x$ is precisely $\mathfrak{g}_x:=\{y\in\mathfrak{g}:[x,y]=0\}$. 

Let $\langle\cdot,\cdot\rangle:\mathfrak{g}\otimes_{\mathbb{C}}\mathfrak{g}\rightarrow\mathbb{C}$ be the Killing form, which one knows to be non-degenerate and $G$-invariant. It follows that
\begin{equation}\label{Equation: Killing isomorphism}\mathfrak{g}\rightarrow\mathfrak{g}^*,\quad x\mapsto\langle x,\cdot\rangle,\quad x\in\mathfrak{g}\end{equation} defines an isomorphism between the adjoint representation and its dual. This isomorphism allows us to understand the moment maps of Hamiltonian $G$-actions as being $\mathfrak{g}$-valued, and it allows us to transfer the canonical Poisson structure on $\mathfrak{g}^*$ to $\mathfrak{g}$. Note that $\mathbb{C}[\mathfrak{g}]$ thereby becomes a Poisson algebra with Poisson centre $\mathbb{C}[\mathfrak{g}]^G$, the algebra of $G$-invariant regular functions on $\mathfrak{g}$. 

\section{The universal centralizer}\label{Section: The universal centralizer}
We now formalize our introductory discussion of the universal centralizer and its properties. This begins in Section \ref{Subsection: Symplectic structure}, which realizes the universal centralizer as a symplectic subvariety of $T^*G$ (Proposition \ref{Proposition: Symplectic subvariety}). Section \ref{Subsection: The integrable system} is devoted to the integrable system on the universal centralizer, and includes results on level sets (Proposition \ref{Proposition: Moment map image}), Hamiltonian vector fields (Proposition \ref{Proposition: Description of Hamiltonian vector fields}), and Carath\'eodory--Jacobi--Lie coordinates (Proposition \ref{Proposition: Surjective local biholomorphism}, Proposition \ref{Proposition: Symplectic forms}, and the last paragraph of \ref{Subsection: The integrable system}).  

\subsection{Symplectic structure}\label{Subsection: Symplectic structure}
Note that the left trivialization of $T^*G$ and the isomorphism \eqref{Equation: Killing isomorphism} determine isomorphisms $$T^*G\cong G\times\mathfrak{g}^*\cong G\times\mathfrak{g}$$ of vector bundles over $G$. The canonical symplectic form on $T^*G$ thereby corresponds to a symplectic form $\Omega$ on $G\times\mathfrak{g}$, defined as follows on each tangent space $T_{(g,x)}(G\times\mathfrak{g})=T_gG\oplus\mathfrak{g}$:
\begin{equation}\label{Equation: Symplectic form on cotangent bundle}\Omega_{(g,x)}\bigg((d_eL_g(y_1),z_1),(d_eL_g(y_2),z_2)\bigg)=\langle y_1,z_2\rangle-\langle y_2,z_1\rangle+\langle x,[y_1,y_2]\rangle,\quad y_1,y_2,z_1,z_2\in\mathfrak{g},\end{equation}
where $L_g:G\rightarrow G$ is left translation by $g$ and $d_eL_g:\mathfrak{g}\rightarrow T_gG$ is the differential of $L_g$ at the identity $e\in G$ (cf. \cite[Section 5, Equation (14L)]{Marsden}).
It is then not difficult to verify that
\begin{subequations}
\begin{align}& h\cdot (g,x):=(hg,x), \quad h\in G,\text{ }(g,x)\in G\times\mathfrak{g}\label{Equation: Left action on first factor},\\
& h\cdot (g,x):=(gh^{-1},\Adj_h(x)), \quad h\in G,\text{ }(g,x)\in G\times\mathfrak{g} \label{Equation: Analogue of right cotangent lift}
\end{align}
\end{subequations}
define commuting Hamiltonian actions of $G$ on $G\times\mathfrak{g}$, and that
\begin{subequations}
\begin{align}
        & \mu_L:G\times\mathfrak{g}\rightarrow\mathfrak{g}, \quad (g,x)\mapsto\Adj_g(x),\quad (g,x)\in G\times\mathfrak{g}\label{Equation: Left moment map},\\
        & \mu_R:G\times\mathfrak{g}\rightarrow\mathfrak{g}, \quad (g,x)\mapsto -x,\quad (g,x)\in G\times\mathfrak{g}\label{Equation: Right moment map}
\end{align}
\end{subequations}
are moment maps for \eqref{Equation: Left action on first factor} and \eqref{Equation: Analogue of right cotangent lift}, respectively.  It follows that 
$$(h_1,h_2)\cdot (g,x):=(h_1gh_2^{-1},\Adj_{h_2}(x)), \quad (h_1,h_2)\in G\times G,\text{ }(g,x)\in G\times\mathfrak{g}$$
defines a Hamiltonian action of $G\times G$ on $G\times\mathfrak{g}$, for which
\begin{equation}\label{Equation: Product moment map}\mu:=(\mu_L,\mu_R):G\times\mathfrak{g}\rightarrow\mathfrak{g}\oplus\mathfrak{g},\quad (g,x)\mapsto(\Adj_g(x),-x),\quad (g,x)\in G\times\mathfrak{g}\end{equation} is necessarily a moment map.    

Now recall the $\mathfrak{sl}_2$-triple $(\xi,h,\eta)$ fixed in \ref{Subsection: Outline of results}, noting that the associated Slodowy slice is 
$$\mathcal{S}:=\xi+\mathfrak{g}_{\eta}\subseteq\mathfrak{g}.$$ This slice is fruitfully studied in relation to $\mathfrak{g}_{\text{reg}}$, the open, dense, $G$-invariant subvariety of $\mathfrak{g}$ defined by 
$$\mathfrak{g}_{\text{reg}}:=\{x\in\mathfrak{g}:\dim(\mathfrak{g}_x)=r\}.$$ One knows that $x\in\mathfrak{g}$ lies in $\mathfrak{g}_{\text{reg}}$ if and only if $x$ is $G$-conjugate to a vector in $\mathcal{S}$. In this case, $x$ is $G$-conjugate to a unique vector in $\mathcal{S}$ (see \cite[Theorem 8]{KostantLie}).  

Now let the universal centralizer $\mathcal{Z}_{\mathfrak{g}}$ be as defined in \eqref{Equation: Universal centralizer}, observing that $\mathcal{Z}_{\mathfrak{g}}$ is a closed subvariety of $G\times\mathfrak{g}$. Using the symplectic form \eqref{Equation: Symplectic form on cotangent bundle}, one can strengthen this observation as follows.   

\begin{proposition}\label{Proposition: Symplectic subvariety}
$\mathcal{Z}_{\mathfrak{g}}$ is a $2r$-dimensional symplectic subvariety of $G\times\mathfrak{g}$.
\end{proposition}

\begin{proof}
Consider the surjective map \begin{equation}\label{Equation: Surjective map}\pi:\mathcal{Z}_{\mathfrak{g}}\rightarrow\mathcal{S},\quad (g,x)\mapsto x,\quad (g,x)\in\mathcal{Z}_{\mathfrak{g}},\end{equation} noting that $\pi^{-1}(x)\cong G_x$ for all $x\in\mathcal{S}$. We thus have $$\dim(\pi^{-1}(x))=\dim(G_x)=\dim(\mathfrak{g}_x)=r$$ for all $x\in\mathcal{S}$, where the last equality is based on the fact that $x\in\mathcal{S}\subseteq\mathfrak{g}_{\text{reg}}$. Note also that $\dim(\mathcal{S})=\dim(\mathfrak{g}_{\eta})=r$, as one knows that $\eta\in\mathfrak{g}_{\text{reg}}$ (see \cite[Theorem 5.3]{Kostant}). Taken together, the previous two sentences imply that $\dim(\mathcal{Z}_{\mathfrak{g}})=2r$.

To prove that $\mathcal{Z}_{\mathfrak{g}}$ is a symplectic subvariety of $G\times\mathfrak{g}$, we first note that $((\xi,-\xi),(h,h),(\eta,-\eta))$ is an $\mathfrak{sl}_2$-triple in $\mathfrak{g}\oplus\mathfrak{g}$. The Slodowy slice associated to this triple is easily seen to be $\mathcal{S}\times(-\mathcal{S})\subseteq\mathfrak{g}\oplus\mathfrak{g}$, where $-\mathcal{S}:=\{-x:x\in\mathcal{S}\}\subseteq\mathfrak{g}$. An application of \cite[Proposition 6]{CrooksVP} then shows $\mu^{-1}(\mathcal{S}\times(-\mathcal{S}))$ to be a (smooth) symplectic subvariety of $G\times\mathfrak{g}$, reducing us to proving $\mathcal{Z}_{\mathfrak{g}}=\mu^{-1}(\mathcal{S}\times(-\mathcal{S}))$.

Suppose that $(g,x)\in\mathcal{Z}_{\mathfrak{g}}$, so that $x\in\mathcal{S}$ and $\Adj_g(x)=x$. The formula \eqref{Equation: Product moment map} then gives $$\mu(g,x)=(\Adj_g(x),-x)=(x,-x)\in\mathcal{S}\times(-\mathcal{S}),$$ establishing the inclusion $\mathcal{Z}_{\mathfrak{g}}\subseteq\mu^{-1}(S\times(-\mathcal{S}))$. Now assume that $(g,x)\in\mu^{-1}(\mathcal{S}\times(-\mathcal{S}))$, i.e. $\Adj_g(x)\in\mathcal{S}$ and $x\in\mathcal{S}$ (see \eqref{Equation: Product moment map}). By the discussion of $\mathcal{S}$ preceding this proposition, we must have $\Adj_g(x)=x$. It follows that $(g,x)\in\mathcal{Z}_{\mathfrak{g}}$, yielding the inclusion $\mu^{-1}(\mathcal{S}\times(-\mathcal{S}))\subseteq\mathcal{Z}_{\mathfrak{g}}$. This completes the proof.
\end{proof}

\begin{remark}
The variety $G\times\mathfrak{g}$ is hyperk\"ahler (see \cite{Kronheimer}), in which context the equation $\mathcal{Z}_{\mathfrak{g}}=\mu^{-1}(\mathcal{S}\times(-\mathcal{S}))$ realizes $\mathcal{Z}_{\mathfrak{g}}$ as a \textit{hyperk\"ahler slice} in $G\times\mathfrak{g}$. The study of hyperk\"ahler slices originates in Bielawski's work \cite{Bielawski}, and we refer the reader to \cite{BielawskiComplex} and \cite[Section 4.2]{CrooksVP} for further details. 
\end{remark}

\subsection{The integrable system}\label{Subsection: The integrable system}
Let $f_1,\ldots,f_r$ be the homogeneous, algebraically independent generators of $\mathbb{C}[\mathfrak{g}]^G$ fixed in \ref{Subsection: Outline of results}, and define $F:\mathfrak{g}\rightarrow\mathbb{C}^r$ by $F=(f_1,\ldots,f_r)$. The restriction of $F$ to $\mathcal{S}$ is then a variety isomorphism (see \cite[Theorem 7]{KostantLie}), to be denoted 
\begin{equation}\label{Equation: Variety isomorphism} F\big\vert_{\mathcal{S}}:\mathcal{S}\xrightarrow{\cong}\mathbb{C}^r,\quad x\mapsto (f_1(x),\ldots,f_r(x)),\quad x\in\mathcal{S}.\end{equation} Let us also define $\widetilde{f}_1,\ldots,\widetilde{f}_r\in\mathbb{C}[\mathcal{Z_{\mathfrak{g}}}]$ by
\begin{equation}\label{Equation: Definition of new f_i}\widetilde{f}_i(g,x)=f_i(x),\quad (g,x)\in\mathcal{Z}_{\mathfrak{g}},\quad i=1,\ldots,r.\end{equation}

\begin{proposition}\label{Proposition: Integrable system}
The functions $\widetilde{f}_1,\ldots,\widetilde{f}_r$ form a completely integrable system on $\mathcal{Z}_{\mathfrak{g}}$.
\end{proposition}

\begin{proof}
The number of functions coincides with $\frac{1}{2}\dim(\mathcal{Z}_{\mathfrak{g}})$ (by Proposition \ref{Proposition: Symplectic subvariety}), so that we are reduced to verifying the following two things: $\widetilde{f}_1,\ldots,\widetilde{f}_r$ Poisson-commute in pairs, and the differentials $d\widetilde{f}_1,\ldots,d\widetilde{f}_r$ are linearly independent at all points in some open dense subset of $\mathcal{Z}_{\mathfrak{g}}$. Accordingly, we recall that $\mathbb{C}[\mathfrak{g}]^G$ is the Poisson centre of $\mathbb{C}[\mathfrak{g}]$. It follows that $f_1,\ldots,f_r$ Poisson-commute in pairs. This implies that the pullback functions $\mu_L^*(f_1),\ldots,\mu_L^*(f_r)$ Poisson-commute in pairs, as the moment map $\mu_L:G\times\mathfrak{g}\rightarrow\mathfrak{g}$ is a Poisson morphism. Since $Z_{\mathfrak{g}}$ is a symplectic subvariety of $G\times\mathfrak{g}$ (see Proposition \ref{Proposition: Symplectic subvariety}), we conclude that the restricted functions $\mu_L^*(f_1)\big\vert_{\mathcal{Z}_{\mathfrak{g}}},\ldots,\mu_L^*(f_r)\big\vert_{\mathcal{Z}_{\mathfrak{g}}}\in\mathbb{C}[\mathcal{Z}_{\mathfrak{g}}]$ must Poisson-commute in pairs. Note also that $$(\mu_L^*(f_i))(g,x)=f_i(\Adj_g(x))=f_i(x)=\widetilde{f}_i(g,x)$$ for all $(g,x)\in\mathcal{Z}_{\mathfrak{g}}$ and $i=1,\ldots,r$, i.e. $\mu_L^*(f_i)\big\vert_{\mathcal{Z}_{\mathfrak{g}}}=\widetilde{f}_i$ for all $i=1,\ldots,r$. These last two sentences imply that $\widetilde{f}_1,\ldots,\widetilde{f}_r$ Poisson-commute in pairs.

To address the linear independence of $d\widetilde{f}_1,\ldots,d\widetilde{f}_r$, we recall that \eqref{Equation: Variety isomorphism} is an isomorphism. It follows that the differentials of $f_1\big\vert_{\mathcal{S}},\ldots,f_r\big\vert_{\mathcal{S}}\in\mathbb{C}[\mathcal{S}]$ are linearly independent at every point of $\mathcal{S}$. Now observe that the map $\pi:\mathcal{Z}_{\mathfrak{g}}\rightarrow\mathcal{S}$ in \eqref{Equation: Surjective map} has a section, defined by sending $x\in\mathcal{S}$ to $(e,x)\in\mathcal{Z}_{\mathfrak{g}}$. We conclude that $\pi$ is a submersion. Together with our discussion of the differentials of $f_1\big\vert_{\mathcal{S}},\ldots,f_r\big\vert_{\mathcal{S}}$, this implies the following: $\pi^*(f_1\big\vert_{\mathcal{S}}),\ldots,\pi^*(f_r\big\vert_{\mathcal{S}})$ have linearly independent differentials at every point in $\mathcal{Z}_{\mathfrak{g}}$. It is also clear that $\pi^*(f_i\big\vert_{\mathcal{S}})=\widetilde{f}_i$ for all $i=1,\ldots,r$, so that our proof complete.
\end{proof} 

A natural next step is to examine the level sets of 
$$\widetilde{F}:=(\widetilde{f}_1,\ldots,\widetilde{f}_r):\mathcal{Z}_{\mathfrak{g}}\rightarrow\mathbb{C}^r,$$
which are described as follows.

 \begin{proposition}\label{Proposition: Moment map image}
The map $\widetilde{F}$ is surjective and $\widetilde{F}^{-1}(\widetilde{F}(g,x))=G_x\times\{x\}$ for all $(g,x)\in\mathcal{Z}_{\mathfrak{g}}$.
\end{proposition}

\begin{proof}
Given $z=(z_1,\ldots,z_r)\in\mathbb{C}^r$, one can use \eqref{Equation: Variety isomorphism} to find $x\in\mathcal{S}$ such that $f_i(x)=z_i$ for all $i=1,\ldots,r$. It follows that $(e,x)\in\mathcal{Z}_{\mathfrak{g}}$ and $\widetilde{F}(e,x)=z$, establishing surjectivity.

Now fix $(g,x)\in\mathcal{Z}_{\mathfrak{g}}$, noting that the inclusion $G_x\times\{x\}\subseteq \widetilde{F}^{-1}(\widetilde{F}(g,x))$ is straightforward. To establish the opposite inclusion, suppose that $(g',x')\in\mathcal{Z}_{\mathfrak{g}}$ satisfies $\widetilde{F}(g',x')=\widetilde{F}(g,x)$. This is equivalent to the statement $f_i(x')=f_i(x)$ for all $i=1,\ldots,r$, which by \eqref{Equation: Variety isomorphism} forces $x=x'$ to hold. We conclude that $g'\in G_{x'}=G_x$, giving $(g',x')\in G_x\times\{x\}$. The inclusion $\widetilde{F}^{-1}(\widetilde{F}(g,x))\subseteq G_x\times\{x\}$ now follows, and the proof is complete. 
\end{proof}

In addition to motivating Proposition \ref{Proposition: Moment map image}, Proposition \ref{Proposition: Integrable system} naturally leads us to study the Hamiltonian vector fields $H(\widetilde{f}_1),\ldots,H(\widetilde{f}_r)$ of $\widetilde{f}_1,\ldots,\widetilde{f}_r$, respectively. We begin this study by fixing $i\in\{1,\ldots,r\}$ and $x\in\mathfrak{g}$. Note that the differential $d_xf_i$ is naturally a vector in $\mathfrak{g}^*$, which the isomorphism \eqref{Equation: Killing isomorphism} identifies with a vector $(d_xf_i)^{\vee}\in\mathfrak{g}$. It is not difficult to verify that $(d_xf_i)^{\vee}\in\mathfrak{g}_x$ (see \cite[Proposition 1.3]{KostantSolution}), implying that $d_eL_g((d_xf_i)^{\vee})\in T_g(G_x)$ for all $g\in G_x$. If we assume that $x\in\mathcal{S}$ and $g\in G_x$, i.e. $(g,x)\in\mathcal{Z}_{\mathfrak{g}}$, then
$$(d_eL_g((d_xf_i)^{\vee}),0)\in T_g(G_x)\oplus\{0\}=T_{(g,x)}(G_x\times\{x\})\subseteq T_{(g,x)}\mathcal{Z}_{\mathfrak{g}}.$$

\begin{proposition}\label{Proposition: Description of Hamiltonian vector fields}
If $(g,x)\in\mathcal{Z}_{\mathfrak{g}}$ and $i\in\{1,\ldots,r\}$, then $$H(\widetilde{f_i})_{(g,x)}=(d_eL_g((d_xf_i)^{\vee}),0).$$
\end{proposition}

\begin{proof}
Noting that $T_{(g,x)}\mathcal{Z}_{\mathfrak{g}}\subseteq T_{(g,x)}(G\times\mathfrak{g})=T_gG\oplus\mathfrak{g}$, any given $v\in T_{(g,x)}\mathcal{Z}_{\mathfrak{g}}$ must have the form $v=(d_eL_g(y),z)$ with $y,z\in\mathfrak{g}$. It then follows immediately from \eqref{Equation: Definition of new f_i} that \begin{equation}\label{Equation: Preliminary} d_{(g,x)}\widetilde{f_i}(v)=d_xf_i(z).\end{equation} Now let $\omega$ and $\Omega$ denote the symplectic forms on $\mathcal{Z}_{\mathfrak{g}}$ and $G\times\mathfrak{g}$, respectively, noting that 
\begin{align*}
\omega_{(g,x)}\bigg((d_eL_g((d_xf_i)^{\vee}),0),v\bigg) & = \omega_{(g,x)}\bigg((d_eL_g((d_xf_i)^{\vee}),0),(d_eL_g(y),z)\bigg)\\
& = \Omega_{(g,x)}\bigg((d_eL_g((d_xf_i)^{\vee}),0),(d_eL_g(y),z)\bigg)\quad\text{(by Proposition \ref{Proposition: Symplectic subvariety})}\\
& = \langle (d_xf_i)^{\vee},z\rangle +\langle x,[(d_xf_i)^{\vee},y]\rangle\quad\hspace{50pt}\text{(by \eqref{Equation: Symplectic form on cotangent bundle})}\\
& = d_xf_i(z) + \langle [x,(d_xf_i)^{\vee}],y\rangle\\
& = d_xf_i(z)\quad\hspace{153pt}\text{(since $(d_xf_i)^{\vee}\in\mathfrak{g}_x$)}\\
& = d_{(g,x)}\widetilde{f_i}(v)\quad\hspace{140pt}\text{(by \eqref{Equation: Preliminary})}.
\end{align*}
This completes the proof.  
\end{proof}     

The Hamiltonian vector fields $H(\widetilde{f}_i)$ are necessarily tangent to the level sets $G_x\times\{x\}\cong G_x$ of $\widetilde{F}$, a consequence of $\widetilde{f}_1,\ldots,\widetilde{f}_r$ forming an integrable system. With this in mind, Proposition \ref{Proposition: Description of Hamiltonian vector fields} implies the following description of $H(\widetilde{f}_i)$ on each level set $G_x$: $H(\widetilde{f}_i)$ is the left-invariant vector field on $G_x$ associated to $(d_xf_i)^{\vee}\in\mathfrak{g}_x$. In other words, the flow of $H(\widetilde{f}_i)$ is given by \begin{equation}\label{Equation: Complex flow}\Phi_i:\mathbb{C}\times\mathcal{Z}_{\mathfrak{g}}\rightarrow\mathcal{Z}_{\mathfrak{g}},\quad (t,(g,x))\mapsto (g\exp(t (d_xf_i)^{\vee}),x),\quad (t,(g,x))\in\mathbb{C}\times\mathcal{Z}_{\mathfrak{g}}.\end{equation}

Consider the biholomorphism $(\Phi_i)_{t}:=\Phi_i(t,\cdot):\mathcal{Z}_{\mathfrak{g}}\rightarrow\mathcal{Z}_{\mathfrak{g}}$ for each fixed $t\in\mathbb{C}$ and $i\in\{1,\ldots,r\}$. By appealing to generalities on Hamiltonian flows or performing a direct calculation, one sees that $(\Phi_i)_t$ is a symplectomorphism. It follows that $$\Phi_{\lambda}:=(\Phi_1)_{\lambda_1}\circ\cdots\circ(\Phi_r)_{\lambda_r}:\mathcal{Z}_{\mathfrak{g}}\rightarrow\mathcal{Z}_{\mathfrak{g}}$$ defines a symplectomorphism for each $\lambda=(\lambda_1,\ldots,\lambda_r)\in\mathbb{C}^r$. Now form the map \begin{equation}\label{Equation: Local biholomorphism}\Phi:\mathbb{C}^r\times\mathcal{S}\rightarrow\mathcal{Z}_{\mathfrak{g}},\quad (\lambda,x)\mapsto \Phi_{\lambda}(e,x),\quad (\lambda,x)\in\mathbb{C}^r\times\mathcal{S}.\end{equation} An application of \eqref{Equation: Complex flow} reveals that
\begin{equation}\label{Equation: Long formula}\Phi(\lambda,x)=\bigg(\exp\bigg(\sum_{i=1}^r\lambda_i(d_xf_i)^{\vee}\bigg),x\bigg),\quad (\lambda,x)\in\mathbb{C}^r\times\mathcal{S}.\end{equation}

\begin{proposition}\label{Proposition: Surjective local biholomorphism}
The map $\Phi:\mathbb{C}^r\times\mathcal{S}\rightarrow\mathcal{Z}_{\mathfrak{g}}$ is a a surjective local biholomorphism.
\end{proposition}

\begin{proof}
We begin with a proof of surjectivity, for which we suppose that $(g,x)\in\mathcal{Z}_{\mathfrak{g}}$. Since $x\in\mathfrak{g}_{\text{reg}}$, we have an algebraic group isomorphism $G_x\cong(\mathbb{C}^{\times})^p\times\mathbb{C}^q$ for some $p,q\geq 0$ with $p+q=r$ (see \cite[Proposition 2.4]{KostantSolution}). It follows that the restricted exponential map $\exp\big\vert_{\mathfrak{g}_x}:\mathfrak{g}_x\rightarrow G_x$ is surjective. We may therefore find $y\in\mathfrak{g}_x$ with $\exp(y)=g$. 

Now note that the differentials of $f_1,\ldots,f_r$ are linearly independent at each point of $\mathfrak{g}_{\text{reg}}$ (see \cite[Theorem 9]{KostantLie}), so that $d_xf_1,\ldots,d_xf_r$ are linearly independent. It follows that $(d_xf_1)^{\vee},\ldots,(d_xf_r)^{\vee}$ are linearly independent elements of $\mathfrak{g}_x$, and a dimension count then implies that $(d_xf_1)^{\vee},\ldots,(d_xf_r)^{\vee}$ form a basis of $\mathfrak{g}_x$. We may therefore find $\lambda_1,\ldots,\lambda_r\in\mathbb{C}$ for which $y=\lambda_1(d_xf_1)^{\vee}+\cdots+\lambda_r(d_xf_r)^{\vee}$. Setting $\lambda:=(\lambda_1,\ldots,\lambda_r)$, an application of \eqref{Equation: Long formula} gives $$\Phi(\lambda,x)=(\exp(y),x)=(g,x).$$ We conclude that $\Phi$ is surjective.

To establish that $\Phi$ is a local biholomorphism, fix a point $(\lambda,x)\in\mathbb{C}^r\times\mathcal{S}$. We have $$T_{(\lambda,x)}(\mathbb{C}^r\times\mathcal{S})=T_{\lambda}\mathbb{C}^r\oplus T_x\mathcal{S}=\mathbb{C}^r\oplus\mathfrak{g}_{\eta},$$ so that the differential of $\Phi$ at $(\lambda,x)$ is a map
$$d_{(\lambda,x)}\Phi:\mathbb{C}^r\oplus\mathfrak{g}_{\eta}\rightarrow T_{\Phi(\lambda,x)}\mathcal{Z}_{\mathfrak{g}}.$$ Letting $e_i\in\mathbb{C}^r$ denote the $i^{\text{th}}$ standard basis vector and noting that the Hamiltonian flows $\Phi_1,\ldots,\Phi_r$ necessarily commute in pairs (by Proposition \ref{Proposition: Integrable system}), we have
\begin{align}
\begin{split}\label{Equation: Hamiltonian calculation}
d_{(\lambda,x)}\Phi(e_i,0) & = \frac{d}{dt}\bigg\vert_{t=0}\Phi(\lambda+te_i,x)\\
& =  \frac{d}{dt}\bigg\vert_{t=0}\bigg(\bigg((\Phi_1)_{\lambda_1}\circ\cdots\circ(\Phi_{i-1})_{\lambda_{i-1}}\circ(\Phi_i)_{\lambda_i+t}\circ(\Phi_{i+1})_{\lambda_{i+1}}\circ\cdots\circ(\Phi_{r})_{\lambda_r}\bigg)(e,x)\bigg)\\
& = \frac{d}{dt}\bigg\vert_{t=0}\bigg(\bigg((\Phi_i)_{t}\circ(\Phi_1)_{\lambda_1}\circ\cdots\circ(\Phi_r)_{\lambda_r}\bigg)(e,x)\bigg)\\
& = \frac{d}{dt}\bigg\vert_{t=0}(\Phi_i)_t(\Phi(\lambda,x))\\
& = H(\widetilde{f_i})_{\Phi(\lambda,x)}.
\end{split}
\end{align}
It follows that for $y=(y_1,\ldots,y_r)\in\mathbb{C}^r$,
\begin{equation}\label{Equation: Differential formula} d_{(\lambda,x)}\Phi(y,0)=\sum_{i=1}^ry_iH(\widetilde{f}_i)_{\Phi(\lambda,x)}.\end{equation}

Now note that $$T_{\Phi(\lambda,x)}\mathcal{Z}_{\mathfrak{g}}\subseteq T_{\Phi(\lambda,x)}(G\times\mathfrak{g})=T_gG\oplus\mathfrak{g},$$ where $g:=\exp\big(\sum_{i=1}^r\lambda_i(d_xf_i)^{\vee}\big)$. We may therefore regard $d_{(\lambda,x)}\Phi$ as taking values in $T_gG\oplus\mathfrak{g}$. With this in mind, the following is a straightforward consequence of \eqref{Equation: Long formula}: if $z\in\mathfrak{g}_{\eta}$, then 
\begin{equation}\label{Equation: Preliminary differential} d_{(\lambda,x)}\Phi(0,z)=(w,z)\in T_gG\oplus\mathfrak{g}\end{equation} for some vector $w\in T_gG$. This combines with \eqref{Equation: Differential formula} to give
$$d_{(\lambda,x)}\Phi(y,z)=\bigg(\sum_{i=1}^ry_iH(\widetilde{f}_i)_{\Phi(\lambda,x)}\bigg)+(w,z).$$

Suppose that $d_{(\lambda,x)}\Phi(y,z)=0$, i.e.
$$\bigg(\sum_{i=1}^ry_iH(\widetilde{f}_i)_{\Phi(\lambda,x)}\bigg)+(w,z)=0\in T_gG\oplus\mathfrak{g}.$$ Proposition \ref{Proposition: Description of Hamiltonian vector fields} implies that the linear combination of $H(\widetilde{f}_1)_{\Phi(\lambda,x)},\ldots, H(\widetilde{f}_r)_{\Phi(\lambda,x)}$ lies in $T_g G\oplus\{0\}\subseteq T_gG\oplus\mathfrak{g}$, so that $z=0$ must hold. The equation \eqref{Equation: Preliminary differential} then gives $w=0$, and hence
$$\sum_{i=1}^ry_iH(\widetilde{f}_i)_{\Phi(\lambda,x)}=0.$$
At the same time, the proof of Proposition \ref{Proposition: Integrable system} explains that the differentials of $\widetilde{f}_1,\ldots,\widetilde{f}_r$ are linearly independent at every point in $\mathcal{Z}_{\mathfrak{g}}$. We conclude that $H(\widetilde{f}_1),\ldots,H(\widetilde{f}_r)$ are linearly independent at every point in $\mathcal{Z}_{\mathfrak{g}}$, which in particular implies that $y_1=\cdots=y_r=0$.

We have shown that $(y,z)=(0,0)$, meaning that $d_{(\lambda,x)}\Phi$ is injective. Since the complex manifolds $\mathbb{C}^r\times\mathcal{S}$ and $\mathcal{Z}_{\mathfrak{g}}$ are both $2r$-dimensional, it follows that $d_{(\lambda,x)}\Phi$ is an isomorphism for all $(\lambda,x)\in\mathbb{C}^r\times\mathcal{S}$. This completes the proof.    
\end{proof}

Now let $z_1,\ldots,z_r$ denote the usual coordinates on $\mathbb{C}^r$, and recall that the $\widehat{f}_i:=f_i\big\vert_{\mathcal{S}}$ form a system of coordinates on $\mathcal{S}$ (see \eqref{Equation: Variety isomorphism}). Abusing notation slightly, we let $z_1,\ldots,z_r,\widehat{f}_1,\ldots,\widehat{f}_r$ denote the induced system of coordinates on $\mathbb{C}^r\times\mathcal{S}$. One immediate observation is that \begin{equation}\label{Equation: Darboux equation}\Phi^*(\widetilde{f}_i)=\widehat{f}_i\end{equation} for all $i\in\{1,\ldots,r\}$. We also note that  
$$\nu:=\sum_{i=1}^r dz_i\wedge d\widehat{f}_i$$
defines a symplectic form on $\mathbb{C}^r\times\mathcal{S}$.

\begin{proposition}\label{Proposition: Symplectic forms}
We have $\Phi^*(\omega)=\nu$, where $\omega$ is the symplectic form on $\mathcal{Z}_{\mathfrak{g}}$.
\end{proposition}

\begin{proof}
Let $\partial f_1,\ldots,\partial f_r$ denote the coordinate vector fields on $\mathcal{S}$ associated to $\widehat{f}_1,\ldots,\widehat{f}_r$, respectively. It follows that for each $(\lambda,x)\in\mathbb{C}^r\times\mathcal{S}$, $\{(e_1,0),\ldots,(e_r,0),(0,(\partial f_1)_x),\ldots,(0,(\partial f_r)_x)\}$ is a basis of $T_{(\lambda,x)}(\mathbb{C}^r\times\mathcal{S})=\mathbb{C}^r\oplus\mathfrak{g}_{\eta}$. Observe also that $\nu_{(\lambda,x)}$ is the unique skew-symmetric bilinear form on $\mathbb{C}^r\oplus\mathfrak{g}_{\eta}$ satisfying the following conditions:
\begin{itemize}
\item $\nu_{(\lambda,x)}\big((e_i,0),(e_j,0)\big)=0$ for all $i,j\in\{1,\ldots,r\}$;
\item $\nu_{(\lambda,x)}\big((e_i,0),(0,(\partial f_j)_x)\big)=\delta_{ij}$ for all $i,j\in\{1,\ldots,r\}$;
\item $\nu_{(\lambda,x)}\big((0,(\partial f_i)_x),(0,(\partial f_j)_x)\big)=0$ for all $i,j\in\{1,\ldots,r\}$. 
\end{itemize}
It will therefore suffice to show that these identities hold after replacing $\nu_{(\lambda,x)}$ with $(\Phi^*(\omega))_{(\lambda,x)}$. In other words, it will suffice to verify the following identities:
\begin{itemize}
\item[(i)] $\omega_{\Phi(\lambda,x)}\big(d_{(\lambda,x)}\Phi(e_i,0),d_{(\lambda,x)}\Phi(e_j,0)\big)=0$ for all $i,j\in\{1,\ldots,r\}$;
\item[(ii)] $\omega_{\Phi(\lambda,x)}\big(d_{(\lambda,x)}\Phi(e_i,0),d_{(\lambda,x)}\Phi(0,(\partial f_j)_x)\big)=\delta_{ij}$ for all $i,j\in\{1,\ldots,r\}$;
\item[(iii)] $\omega_{\Phi(\lambda,x)}\big(d_{(\lambda,x)}\Phi(0,(\partial f_i)_x),d_{(\lambda,x)}\Phi(0,(\partial f_j)_x)\big)=0$ for all $i,j\in\{1,\ldots,r\}$. 
\end{itemize}

To address (i), note that \eqref{Equation: Hamiltonian calculation} gives 
$$\omega_{\Phi(\lambda,x)}\big(d_{(\lambda,x)}\Phi(e_i,0),d_{(\lambda,x)}\Phi(e_j,0)\big) = \omega_{\Phi(\lambda,x)}\big(H(\widetilde{f}_i)_{\Phi(\lambda,x)},H(\widetilde{f}_j)_{\Phi(\lambda,x)}\big).$$
The right hand side is necessarily zero for all $i,j\in\{1,\ldots,r\}$, as $\widetilde{f}_1,\ldots,\widetilde{f}_r$ Poisson-commute in pairs (see Proposition \ref{Proposition: Integrable system}). This verifies (i).

We begin the proof of (ii) with the the following calculation:
\begin{align*}
\omega_{\Phi(\lambda,x)}\big(d_{(\lambda,x)}\Phi(e_i,0),d_{(\lambda,x)}\Phi(0,(\partial f_j)_x)\big) & = \omega_{\Phi(\lambda,x)}\big(H(\widetilde{f}_i)_{\Phi(\lambda,x)},d_{(\lambda,x)}\Phi(0,(\partial f_j)_x)\big) \hspace{15pt} \text{(by \eqref{Equation: Hamiltonian calculation})}\\
& = d_{\Phi(\lambda,x)}\widetilde{f}_i\big(d_{(\lambda,x)}\Phi(0,(\partial f_j)_x)\big)\\
& = d_{(\lambda,x)}(\widetilde{f}_i\circ\Phi)(0,(\partial f_j)_x)\\
& = d_{(\lambda,x)}\widehat{f}_i(0,(\partial f_j)_x)\quad\hspace{103pt}\text{(by \eqref{Equation: Darboux equation})},\\
\end{align*}
where $d_{(\lambda,x)}\widehat{f}_i$ is the differential of $\widehat{f}_i$ as a coordinate on $\mathbb{C}^r\times\mathcal{S}$. It is straightforward to see that the last line equals $d_x\widehat{f}_i((\partial f_j)_x)$, where $d_x\widehat{f}_i$ is the differential of the coordinate $\widehat{f}_i$ on $\mathcal{S}$. Since $d_x\widehat{f}_i((\partial f_j)_x)=\delta_{ij}$, we have shown that (ii) holds. 

Before proving (iii), we note the following straightforward consequence of the definition \eqref{Equation: Local biholomorphism}: \begin{equation}\label{Equation: Quick identity} d_{(\lambda,x)}\Phi(0,(\partial f_i)_x)=d_{(e,x)}\Phi_{\lambda}(0,(\partial f_i)_x)\end{equation} for all $i\in\{1,\ldots,r\}$. Note that the vector $(0,(\partial f_i)_x)$ appearing on the right hand side belongs to $T_{(e,x)}\mathcal{Z}_{\mathfrak{g}}\subseteq T_{(e,x)}(G\times\mathfrak{g})=\mathfrak{g}\oplus\mathfrak{g}$, while on the left hand side we have $(0,(\partial f_i)_x)\in T_{(\lambda,x)}(\mathbb{C}^r\times\mathcal{S})=\mathbb{C}^r\oplus\mathfrak{g}_{\eta}$.

Using \eqref{Equation: Quick identity} and recalling that $\Phi_{\lambda}:\mathcal{Z}_{\mathfrak{g}}\rightarrow\mathcal{Z}_{\mathfrak{g}}$ is a symplectomorphism, we obtain
\begin{align*}
\omega_{\Phi(\lambda,x)}\big(d_{(\lambda,x)}\Phi(0,(\partial f_i)_x),d_{(\lambda,x)}\Phi(0,(\partial f_j)_x)\big) & = \omega_{\Phi_{\lambda}(e,x)}\big(d_{(e,x)}\Phi_{\lambda}(0,(\partial f_i)_x),d_{(e,x)}\Phi_{\lambda}(0,(\partial f_j)_x)\big)\\
& = (\Phi_{\lambda}^*(\omega))_{(e,x)}\big((0,(\partial f_i)_x),(0,(\partial f_j)_x)\big)\\
& = \omega_{(e,x)}\big((0,(\partial f_i)_x),(0,(\partial f_j)_x)\big).\\
\end{align*}
Now recall that $(\mathcal{Z}_{\mathfrak{g}},\omega)$ is a symplectic subvariety of $(G\times\mathfrak{g},\Omega)$, meaning that the last line is 
$$\Omega_{(e,x)}\big((0,(\partial f_i)_x),(0,(\partial f_j)_x)\big).$$ This is zero by virtue of the formula \eqref{Equation: Symplectic form on cotangent bundle}, proving that (iii) holds.    
\end{proof}

\begin{remark}\label{Remark: CJL coordinates}
Proposition \ref{Proposition: Symplectic forms} affords us a brief proof that $\Phi$ is a local biholomorphism. One begins by noting that $\nu$ is non-degenerate, and that the dimensions of $\mathbb{C}^r\times\mathcal{S}$ and $\mathcal{Z}_{\mathfrak{g}}$ coincide. Together with the identity $\Phi^*(\omega)=\nu$ from Proposition \ref{Proposition: Symplectic forms}, these observations force $d_{(\lambda,x)}\Phi$ to be an isomorphism for all $(\lambda,x)\in\mathbb{C}^r\times\mathcal{S}$. The proof of Proposition \ref{Proposition: Surjective local biholomorphism} establishes this fact via more conventional arguments.     
\end{remark}

Now fix any point $x\in\mathcal{Z}_{\mathfrak{g}}$. Proposition \ref{Proposition: Surjective local biholomorphism} then implies that $\Phi$ restricts to a biholomorphism $U\rightarrow V$, where $V$ is a suitably chosen open neighbourhood of $x$ in $\mathcal{Z}_{\mathfrak{g}}$ and $U$ is an appropriate open subset of $\mathbb{C}^r\times\mathcal{S}$. The coordinates $z_1,\ldots,z_r,\widehat{f}_1,\ldots,\widehat{f}_r$ on $\mathbb{C}^r\times\mathcal{S}$ thereby induce coordinates on $V$. By \eqref{Equation: Darboux equation} and Proposition \ref{Proposition: Symplectic forms}, these latter coordinates satisfy the conclusion of the Carath\'eodory--Jacobi--Lie theorem for completely integrable systems (e.g. \cite[Theorem 4.1.8]{Ortega}). A less formal statement is that we have produced explicit \textit{Carath\'eodory--Jacobi--Lie coordinates} for the integrable system on $\mathcal{Z}_{\mathfrak{g}}$.

\section{The Kostant--Toda lattice}\label{Section: The Kostant--Toda lattice}
We now discuss the second integrable system featuring in our paper --- the Kostant--Toda lattice. The discussion begins in \ref{Subsection: The basics}, which recalls the details underlying Kostant's construction of the Kostant--Toda lattice. Section \ref{Subsection: Some technical results} subsequently introduces the open set $\mathcal{V}$ from Theorem \ref{Theorem: Main theorem} (see \eqref{Equation: Definition of open Kostant-Toda}), and in addition studies certain holomorphic maps defined on $\mathcal{V}$. Using the setup developed in \ref{Subsection: Some technical results}, Section \ref{Subsection: Level sets} recalls Kostant's description of level sets in the Kostant--Toda lattice (Theorems \ref{Theorem: Kostant fibre isomorphism} and \ref{Theorem: Kostant Hamiltonian}). Section \ref{Subsection: A holomorphic map} then harnesses this description to define and study a holomorphic map $\mathcal{V}\rightarrow G^*$ (see \eqref{Equation: Definition of lambda}), where $G^*$ is an explicit open dense subset of $G$ (see \eqref{Equation: Translate}).  

\subsection{The basics}\label{Subsection: The basics}
Recall the notation and conventions set in Section \ref{Section: Notation and conventions}. Note that our Cartan subalgebra $\mathfrak{t}$ and (positive) simple roots $\Delta$ determine the following data: roots $\mathcal{R}\subseteq\mathfrak{t}^*$, positive roots $\mathcal{R}_{+}\subseteq\mathcal{R}$, negative roots $\mathcal{R}_{-}\subseteq\mathcal{R}$, a positive Borel subalgebra $\mathfrak{b}\subseteq\mathfrak{g}$, and a negative Borel subalgebra $\mathfrak{b}_{-}\subseteq\mathfrak{g}$. Let $\mathfrak{u}$ and $\mathfrak{u}_{-}$ be the nilpotent radicals of $\mathfrak{b}$ and $\mathfrak{b}_{-}$, respectively. It follows that  $$\mathfrak{b}=\mathfrak{t}\oplus\bigoplus_{\alpha\in\mathcal{R}_{+}}\mathfrak{g}_{\alpha},\quad \mathfrak{u}=\bigoplus_{\alpha\in\mathcal{R}_{+}}\mathfrak{g}_{\alpha},\quad \mathfrak{b}_{-}=\mathfrak{t}\oplus\bigoplus_{\alpha\in\mathcal{R}_{-}}\mathfrak{g}_{\alpha},\quad\text{and}\quad \mathfrak{u}_{-}=\bigoplus_{\alpha\in\mathcal{R}_{-}}\mathfrak{g}_{\alpha},$$
where $$\mathfrak{g}_{\alpha}:=\{y\in\mathfrak{g}:[x,y]=\alpha(x)y\text{ for all }x\in\mathfrak{t}\}$$ is the root space associated to $\alpha\in\mathcal{R}$.

Let $T$, $B$, $B_{-}$, $U$, and $U_{-}$ denote the closed, connected subgroups of $G$ with respective Lie algebras $\mathfrak{t}$, $\mathfrak{b}$, $\mathfrak{b}_{-}$, $\mathfrak{u}$, and $\mathfrak{u}_{-}$. We then have a Weyl group $W:=N_G(T)/T$ and a weight lattice $\Lambda:=\mathrm{Hom}(T,\mathbb{C}^{\times})$ of all algebraic group morphisms $T\rightarrow\mathbb{C}^{\times}$. There is a canonical identification of $\Lambda$ with a $\mathbb{Z}$-lattice in $\mathfrak{t}^*$, and we shall make no distinction between elements of $\Lambda$ and those of the aforementioned lattice. Each $\alpha\in\mathcal{R}$ may thereby be considered an algebraic group morphism $T\rightarrow\mathbb{C}^{\times}$, allowing us to write
$$\mathfrak{g}_{\alpha}=\{y\in\mathfrak{g}:\Adj_t(y)=\alpha(t)y\text{ for all }t\in T\}.$$

Now let $\mathcal{O}_{\text{KT}}\subseteq\mathfrak{g}$ be as defined in \eqref{Equation: Kostant--Toda definition}, and let $\overline{f}_1,\ldots,\overline{f}_r:\mathcal{O}_{\text{KT}}\rightarrow\mathbb{C}$ denote the restrictions of $f_1,\ldots,f_r$ to $\mathcal{O}_{\text{KT}}$, respectively. With these things in mind, we summarize the results \cite[Proposition 2.3.1, Proposition 6.4]{KostantSolution} and \cite[Proposition 4, Theorem 29]{KostantFlag} as follows.

\begin{theorem}\label{Theorem: Kostant's theorem}
The variety $\mathcal{O}_{\emph{KT}}$ carries a canonical symplectic form, and $\overline{f}_1,\ldots,\overline{f}_r$ form a completely integrable system on $\mathcal{O}_{\emph{KT}}$ with respect to this symplectic form.
\end{theorem}

We use the term \textit{Kostant--Toda lattice} for the integrable system in Theorem \ref{Theorem: Kostant's theorem}.  In the interest of making a simple observation about this system, we set $\overline{F}:=(\overline{f}_1,\ldots,\overline{f}_r):\mathcal{O}_{\text{KT}}\rightarrow\mathbb{C}^r$. The observation is then as follows.

\begin{lemma}\label{Lemma: Precisely conjugate}
If $x\in\mathcal{O}_{\emph{KT}}$, then $\overline{F}^{-1}(\overline{F}(x))$ consists precisely of those points in $\mathcal{O}_{\emph{KT}}$ that are $G$-conjugate to $x$.
\end{lemma}

\begin{proof}
Fix $x\in\mathcal{O}_{\text{KT}}$. Since $\xi+\mathfrak{b}\subseteq\mathfrak{g}_{\text{reg}}$ (see \cite[Lemma 10]{KostantLie}) and $\mathcal{O}_{\text{KT}}\subseteq\xi+\mathfrak{b}$, we have the inclusion $\mathcal{O}_{\text{KT}}\subseteq\mathfrak{g}_{\text{reg}}$. It then follows from \cite[Theorem 2]{KostantLie} that $y\in\mathcal{O}_{\text{KT}}$ is $G$-conjugate to $x$ if and only if $f_i(x)=f_i(y)$ for all $i\in\{1,\ldots,r\}$. The latter condition is precisely the statement that $\overline{F}(x)=\overline{F}(y)$, i.e. $y\in\overline{F}^{-1}(\overline{F}(x))$. This completes the proof.  
\end{proof}

\subsection{Some technical results}\label{Subsection: Some technical results}
We now introduce two important subsets of $\mathfrak{t}$. The first is the open subset $\mathfrak{t}_{\text{reg}}:=\mathfrak{g}\cap\mathfrak{g}_{\text{reg}}$, which is also described by $$\mathfrak{t}_{\text{reg}}=\{x\in\mathfrak{t}:\alpha(x)\neq 0\text{ for all }\alpha\in\mathcal{R}\}.$$ The second is the open subset of $\mathfrak{t}$ defined by
$$\mathcal{C}:=\{x\in\mathfrak{t}:\mathrm{Re}(\alpha(x))>0\text{ for all }\alpha\in\Delta\}.$$ Given $x\in\mathcal{C}$, one readily verifies that $\mathrm{Re}(\alpha(x))>0$ for all $\alpha\in\mathcal{R}_{+}$ and $\mathrm{Re}(\alpha(x))<0$ for all $\alpha\in\mathcal{R}_{-}$. This establishes that $\alpha(x)\neq 0$ for all $\alpha\in\mathcal{R}$, so that $\mathcal{C}\subseteq\mathfrak{t}_{\text{reg}}$.

Equipped with the preceding discussion, we may describe the image $\mathcal{D}:=F(\mathcal{C})$ of $\mathcal{C}$ under $F=(f_1,\ldots,f_r):\mathfrak{g}\rightarrow\mathbb{C}^r$.

\begin{lemma}\label{Lemma: Open dense}
The subset $\mathcal{D}$ is open and dense in $\mathbb{C}^r$.
\end{lemma}

\begin{proof}
We begin by recalling the following fundamental domain $\mathcal{C}_{\mathfrak{t}}\subseteq\mathfrak{t}$ for the action of $W$ on $\mathfrak{t}$ (see \cite[Section 2.2]{Collingwood}):
$$\mathcal{C}_{\mathfrak{t}}:=\bigg\{x\in\mathfrak{t}:\forall\alpha\in\Delta,\text{ }\mathrm{Re}(\alpha(x))\geq 0 \text{ and }\\ \bigg(\mathrm{Re}(\alpha(x))=0\Rightarrow\mathrm{Im}(\alpha(x))\geq 0\bigg)\bigg\}.$$ Now let $\overline{\mathcal{C}}$ denote the closure of $\mathcal{C}$ in $\mathfrak{g}$, and observe that $\mathcal{C}_{\mathfrak{t}}\subseteq\overline{\mathcal{C}}$. The continuity of $F$ then implies that $$F(\mathcal{C}_{\mathfrak{t}})\subseteq F(\overline{\mathcal{C}})\subseteq \overline{F(\mathcal{C})}=\overline{\mathcal{D}},$$ with the last two closures taken in $\mathbb{C}^r$. We also know that $$F(\mathcal{C}_{\mathfrak{t}})=F(\mathfrak{t})=\mathbb{C}^r,$$ where the second equality is established in the proof of \cite[Proposition 10]{KostantLie} and the first equality comes from the following two things: the $W$-invariance of $F\big\vert_{\mathfrak{t}}$ and the fact that $\mathcal{C}_{\mathfrak{t}}$ is a fundamental domain. We conclude that $\overline{\mathcal{D}}=\mathbb{C}^r$, i.e. $\mathcal{D}$ is dense in $\mathbb{C}^r$.

We now prove that $\mathcal{D}$ is open. Accordingly, note that $F\big\vert_{\mathfrak{g}_{\text{reg}}}:\mathfrak{g}_{\text{reg}}\rightarrow\mathbb{C}^r$ is a submersion (see \cite[Theorem 9]{KostantLie}) whose level sets are precisely the $G$-orbits in $\mathfrak{g}_{\text{reg}}$ (see \cite[Theorem 3]{KostantLie}). It follows that $\mathrm{ker}(d_x F)=T_x(Gx)$ for all $x\in\mathfrak{g}_{\text{reg}}$. At the same time, it is not difficult to verify that $T_x(\mathfrak{t}_{\text{reg}})$ and $T_x(Gx)$ are complementary subspaces of $\mathfrak{g}$ for all $x\in\mathfrak{t}_{\text{reg}}$. These last two sentences imply that $F\big\vert_{\mathfrak{t}_{\text{reg}}}:\mathfrak{t}_{\text{reg}}\rightarrow\mathbb{C}^r$ is a local biholomorphism. Since $\mathcal{C}$ is an open subset of $\mathfrak{t}_{\text{reg}}$, it follows that $F(\mathcal{C})=\mathcal{D}$ is open in $\mathbb{C}^r$. 
\end{proof}  

We now consider the open subset of $\mathcal{O}_{\mathrm{KT}}$ given by \begin{equation}\label{Equation: Definition of open Kostant-Toda}\mathcal{V}:=\overline{F}^{-1}(\mathcal{D})=F^{-1}(\mathcal{D})\cap\mathcal{O}_{\text{KT}},\end{equation} about which the following is true.

\begin{proposition}\label{Proposition: Dense}
The open set $\mathcal{V}$ is dense in $\mathcal{O}_{\emph{KT}}$.
\end{proposition}

\begin{proof}
It follows from \cite[Proposition 2.3.1]{KostantSolution} that $\overline{F}:\mathcal{O}_{\text{KT}}\rightarrow\mathbb{C}^r$ is a submersion, and hence also an open map. In particular, the preimage of a dense subset under $\overline{F}$ is necessarily dense. Lemma \ref{Lemma: Open dense} and \eqref{Equation: Definition of open Kostant-Toda} now combine to imply the desired result.
\end{proof}

We devote the balance of \ref{Subsection: Some technical results} to the study of two holomorphic maps on $\mathcal{V}$. The first such map involves the translated subset $\xi+\mathcal{C}\subseteq\mathfrak{g}$ and is described as follows.

\begin{lemma}\label{Lemma: Conjugate}
There is a holomorphic map $\theta:\mathcal{V}\rightarrow\xi+\mathcal{C}$ with the following property: if $x\in\mathcal{V}$, then $\theta(x)$ is the unique element of $\xi+\mathcal{C}$ that is $G$-conjugate to $x$.
\end{lemma}

\begin{proof}
By \cite[Corollary 3.1.43]{Chriss}, we have $$F(\xi+\mathcal{C})=F(\mathcal{C})=\mathcal{D}.$$ It follows that $F$ restricts to a surjective holomorphic map $F_{\mathcal{C}}:\xi+\mathcal{C}\rightarrow\mathcal{D}$, and we claim that $F_{\mathcal{C}}$ is a biholomorphism. Note that it suffices to establish injectivity.

Suppose that $x_1,x_2\in\mathcal{C}$ satisfy $F(\xi+x_1)=F(\xi+x_2)$. It then follows from \cite[Corollary 3.1.43]{Chriss} that $F(x_1)=F(x_2)$. Since $x_1,x_2\in\mathfrak{g}_{\text{reg}}$, this implies that $x_1$ and $x_2$ are $G$-conjugate (see \cite[Theorem 3]{KostantLie}). Note also that $x_1,x_2\in\mathfrak{t}$, so that they must actually be $W$-conjugate. At the same time, $x_1$ and $x_2$ belong to the fundamental domain $\mathcal{C}_{\mathfrak{t}}$ introduced in the proof of Lemma \ref{Lemma: Open dense}. We conclude that $x_1=x_2$, proving injectivity. As discussed above, this shows $F_{\mathcal{C}}$ to be a biholomorphism.  

Let us now consider the holomorphic map $$\theta:\mathcal{V}\rightarrow\xi+\mathcal{C},\quad x\mapsto F_{\mathcal{C}}^{-1}(F(x)),\quad x\in\mathcal{V}.$$  Note that for all $x\in\mathcal{V}$, we have $F(\theta(x))=F(x)$. Since $x,\theta(x)\in\mathfrak{g}_{\text{reg}}$ (by \cite[Lemma 10]{KostantLie}), we conclude that $x$ and $\theta(x)$ are $G$-conjugate (see \cite[Theorem 3]{Kostant}). Now let $y\in\xi+\mathcal{C}$ be $G$-conjugate to $x$, observing that $F(y)=F(x)=F(\theta(x))$. It follows that $y=\theta(x)$, as $F_{\mathcal{C}}$ is a bijection. This completes the proof.
\end{proof}

To construct a second holomorphic map on $\mathcal{V}$, we recall Kostant's variety isomorphism
\begin{equation}\label{Equation: Kostant variety isomorphism}\psi: U\times\mathcal{S}\xrightarrow{\cong} \xi+\mathfrak{b},\quad (u,x)\mapsto\Adj_u(x),\quad (u,x)\in U\times\mathcal{S}\end{equation} (see \cite[Theorem 1.2]{KostantWhittaker}). Note that $\mathcal{V}$ and $\xi+\mathcal{C}$ are both subsets of $\xi+\mathfrak{b}$, so that we may apply $\psi^{-1}$ to elements of $\mathcal{V}$ and $\xi+\mathcal{C}$. Letting $\pi_U:U\times\mathcal{S}\rightarrow U$ be the projection to the first factor, we define our second holomorphic map on $\mathcal{V}$ by
$$\nu:\mathcal{V}\rightarrow U,\quad x\mapsto \pi_U\bigg(\psi^{-1}(x)\bigg)\pi_U\bigg(\psi^{-1}(\theta(x))\bigg)^{-1},\quad x\in\mathcal{V}.$$ 

\begin{lemma}\label{Lemma: Description of nu}
If $x\in\mathcal{V}$, then $\nu(x)$ is the unique element of $U$ satisfying $\Adj_{\nu(x)}(\theta(x))=x$.
\end{lemma}

\begin{proof}
Let us write $\psi^{-1}(x)=(u_1,\overline{x})$ and $\psi^{-1}(\theta(x))=(u_2,\overline{\theta(x)})$, so that \begin{equation}\label{Equation: U-conjugacy}x=\Adj_{u_1}(\overline{x})\quad\text{and}\quad\theta(x)=\Adj_{u_2}(\overline{\theta(x)}).\end{equation} Kostant's result \cite[Theorem 8]{KostantLie} then implies that $\overline{x}$ (resp. $\overline{\theta(x)}$) is the unique element of $\mathcal{S}$ with the property of being $G$-conjugate to $x$ (resp. $\theta(x)$). Since $x$ and $\theta(x)$ are themselves $G$-conjugate (by Lemma \ref{Lemma: Conjugate}), it follows that $\overline{x}=\overline{\theta(x)}$. We may therefore use \eqref{Equation: U-conjugacy} to conclude that $$x=\Adj_{u_1u_2^{-1}}(\theta(x)).$$ Now observe that $u_1=\pi_U(\psi^{-1}(x))$ and $u_2=\pi_U(\psi^{-1}(\theta(x)))$, so that $$u_1u_2^{-1}= \pi_U\bigg(\psi^{-1}(x)\bigg)\pi_U\bigg(\psi^{-1}(\theta(x))\bigg)^{-1}=\nu(x)$$ and
$$x=\Adj_{\nu(x)}(\theta(x)).$$ The uniqueness of $\nu(x)$ is a consequence of \cite[Proposition 2.3.2]{KostantSolution}.
\end{proof}

\subsection{Level sets}\label{Subsection: Level sets}
The map $\theta:\mathcal{V}\rightarrow\xi+\mathcal{C}$ facilitates a convenient description of Kostant's results \cite{KostantSolution} on the level sets $\overline{F}^{-1}(z)$. To this end, let $w_0\in W$ denote the longest element of the Weyl group and recall the root vectors fixed in \ref{Subsection: Outline of results}. Note that there exists a unique lift $\widetilde{w_0}\in N_G(T)$ of $w_0$ that satisfies $$\Adj_{\widetilde{w_0}}(e_{\alpha})=e_{w_0\alpha}$$ for all $\alpha\in\Delta$. Now consider the open subset $U_{-}TU$ of $G$, as well as its left $\widetilde{w_0}$-translate
\begin{equation}\label{Equation: Translate} G^*:=\widetilde{w_0}U_{-}TU\subseteq G.\end{equation} The map $$U_{-}\times T\times U\rightarrow G^*,\quad (u_{-},t,u)\mapsto \widetilde{w_0}u_{-}tu,\quad (u_{-},t,u)\in U_{-}\times T\times U$$ is then a variety isomorphism. Let us invert this isomorphism and project onto the factors $U_{-}$, $T$, and $U$, thereby obtaining variety morphisms
$$\sigma_{U_{-}}:G^*\rightarrow U_{-},\quad \sigma_T:G^*\rightarrow T,\quad\text{and}\quad \sigma_{U}:G^*\rightarrow U.$$ Setting $G_x^*:=G_x\cap G^*$ for all $x\in\mathfrak{g}$, one has the following rephrased version of \cite[Theorem 2.6]{KostantSolution}.  

\begin{theorem}\label{Theorem: Kostant fibre isomorphism}
If $x\in\mathcal{V}$, then 
$$\tau_{\theta(x)}:G_{\theta(x)}^*\rightarrow \overline{F}^{-1}(\overline{F}(x)),\quad g\mapsto \Adj_{\sigma_U(g)}(\theta(x)),\quad g\in G_{\theta(x)}^*$$ defines an isomorphism of varieties.
\end{theorem}

Now let $H(\overline{f}_i)$ be the Hamiltonian vector field on $\mathcal{O}_{\text{KT}}$ determined by $\overline{f}_i$, $i\in\{1,\ldots,r\}$. Since $\overline{f}_1,\ldots,\overline{f}_r$ form a completely integrable system on $\mathcal{O}_{\text{KT}}$, one knows that $H(\overline{f}_i)$ is tangent to $\overline{F}^{-1}(\overline{F}(x))$ for all $x\in\mathcal{O}_{\text{KT}}$. It follows that for each $x\in\mathcal{V}$, $\tau_{\theta(x)}$ identifies $H(\overline{f}_i)$ with a vector field on $G_{\theta(x)}^*$. The latter vector field turns out to be the restriction of the following vector field on $G_{\theta(x)}$: the left-invariant vector field corresponding to $(d_{\theta(x)}f_i)^{\vee}\in\mathfrak{g}_{\theta(x)}$, where the notation $(d_{\theta(x)}f_i)^{\vee}$ is as defined in \ref{Subsection: The integrable system}. This discussion can be formulated as follows (cf. {\cite[Theorem 4.3]{KostantSolution}}).

\begin{theorem}\label{Theorem: Kostant Hamiltonian}
We have \begin{equation}\label{Equation: Relation between vector fields} d_{g}\tau_{\theta(x)}(d_eL_g((d_{\theta(x)}f_i)^{\vee}))=H(\overline{f}_i)_{\tau_{\theta(x)}(g)}\end{equation} for all $x\in\mathcal{V}$ and $g\in G_{\theta(x)}^*$.
\end{theorem}

\subsection{A holomorphic map $\lambda:\mathcal{V}\rightarrow G^*$}\label{Subsection: A holomorphic map}
Let us consider the map \begin{equation}\label{Equation: Definition of lambda}\lambda:\mathcal{V}\rightarrow G^*,\quad x\mapsto\tau_{\theta(x)}^{-1}(x),\quad x\in\mathcal{V}.\end{equation} Theorem \ref{Theorem: Kostant fibre isomorphism} then implies that for all $x\in\mathcal{V}$, $\lambda(x)$ is the unique element of $G_{\theta(x)}^*$ such that \begin{equation}\label{Equation: Preceding}x=\Adj_{\sigma_U(\lambda(x))}(\theta(x)).\end{equation} The following two additional facts about $\lambda$ are needed in Section \ref{Section: The relationship}.

\begin{lemma}\label{Lemma: Property of lambda}
If $x\in\mathcal{V}$, then $\lambda(y)=\tau_{\theta(x)}^{-1}(y)$ for all $y\in \overline{F}^{-1}(\overline{F}(x))$.
\end{lemma}

\begin{proof}
Lemma \ref{Lemma: Precisely conjugate} implies that $x$ and $y$ are $G$-conjugate, which by Lemma \ref{Lemma: Conjugate} means that $\theta(x)=\theta(y)$. We therefore have $\lambda(y)=\tau_{\theta(y)}^{-1}(y)=\tau_{\theta(x)}^{-1}(y)$, where the first equality follows from the definition \eqref{Equation: Definition of lambda}.  
\end{proof}

\begin{lemma}
The map $\lambda:\mathcal{V}\rightarrow G^*$ is holomorphic.
\end{lemma}

\begin{proof}
Given $x\in\mathcal{V}$, we may write $\lambda(x)=\widetilde{w_0}u_{-}tu$ with $u_{-}\in U_{-}$, $t\in T$, and $u\in U$. Let us express $u_{-}$, $t$, and $u$ as functions of $x$. To this end, note that $u=\sigma_U(\lambda(x))$. Equation \eqref{Equation: Preceding} is then $x=\Adj_u(\theta(x))$, which by Lemma \ref{Lemma: Description of nu} implies that \begin{equation}\label{Equation: Expression for u}u=\nu(x).\end{equation}

To express $t$ in terms of $x$, we first recall that $\lambda(x)\in G_{\theta(x)}$. Hence
\begin{equation}\label{Equation: Stabilizer}\theta(x)=\Adj_{\lambda(x)}(\theta(x))=\Adj_{\widetilde{w_0}u_{-}tu}(\theta(x))=\Adj_{\widetilde{w_0}u_{-}t\nu(x)}(\theta(x))=\Adj_{\widetilde{w_0}u_{-}t}(x),\end{equation} where the last equality follows from Lemma \ref{Lemma: Description of nu}. Now apply \eqref{Equation: Kostant--Toda definition} and the fact that $\theta(x)\in\xi+\mathcal{C}$ to write $$x=\xi+x_0+\sum_{\alpha\in\Delta}x_{\alpha}e_{\alpha}\quad\text{ and }\quad\theta(x)=\xi+y$$ with $x_0\in\mathfrak{t}$, $y\in\mathcal{C}$, and $x_{\alpha}\in\mathbb{C}^{\times}$ for all $\alpha\in\Delta$. Observe that \eqref{Equation: Stabilizer} then becomes
\begin{align*}\xi+y & = \Adj_{\widetilde{w_0}u_{-}t}\bigg(\xi+x_0+\sum_{\alpha\in\Delta}x_{\alpha}e_{\alpha}\bigg)
\\ & = \Adj_{\widetilde{w_0}u_{-}}\bigg(\Adj_t(\xi)+x_0+\sum_{\alpha\in\Delta}(\alpha(t)x_{\alpha})e_{\alpha}\bigg)
\\ & = \Adj_{\widetilde{w_0}}\bigg(z + \sum_{\alpha\in\Delta}(\alpha(t)x_{\alpha})e_{\alpha}\bigg)
\\ & = \Adj_{\widetilde{w_0}}(z)+\sum_{\alpha\in\Delta}(\alpha(t)x_{\alpha})e_{w_0\alpha} 
\end{align*}
for some $z\in\mathfrak{b}_{-}$. Note that $y$ and $\Adj_{\widetilde{w_0}}(z)$ belong to $\mathfrak{b}$, while each of $\xi$ and $\sum_{\alpha\in\Delta}(\alpha(t)x_{\alpha})e_{w_0\alpha}$ is a sum of negative simple root vectors. It follows that $$\xi=\sum_{\alpha\in\Delta}(\alpha(t)x_{\alpha})e_{w_0\alpha},$$ or equivalently that $\alpha(t)=x_{\alpha}^{-1}$ for all $\alpha\in\Delta$. Expressing this in terms of the algebraic group isomorphism $$\gamma:T\rightarrow(\mathbb{C}^{\times})^{\Delta},\quad s\mapsto (\alpha(s))_{\alpha\in\Delta},\quad s\in T,$$
we have \begin{equation}\label{Equation: Expression for t}t=\gamma^{-1}((x_{\alpha}^{-1})_{\alpha\in\Delta}).\end{equation}

It remains to write $u_{-}$ in terms of $x$. Accordingly, it is not difficult to verify that
$$\Adj_{\widetilde{w_0}t}(x)\in\mathcal{O}_{\text{KT}}.$$ We also have
$$F(\Adj_{\widetilde{w_0}t}(x))=F(x)\in\mathcal{D},$$ implying that $\Adj_{\widetilde{w_0}t}(x)\in\mathcal{V}$. An application of Lemma \ref{Lemma: Conjugate} then reveals that $\theta(\Adj_{\widetilde{w_0}t}(x))=\theta(x)$. In addition, $\widetilde{w_0}u_{-}\widetilde{w_0}^{-1}\in U$ and \eqref{Equation: Stabilizer} implies that
$$\Adj_{\widetilde{w_0}u_{-}\widetilde{w_0}^{-1}}\bigg(\Adj_{\widetilde{w_0}t}(x)\bigg)=\Adj_{\widetilde{w_0}u_{-}t}(x)=\theta(x)=\theta(\Adj_{\widetilde{w_0}t}(x))$$ It now follows from Lemma \ref{Lemma: Description of nu} that 
$$\widetilde{w_0}u_{-}\widetilde{w_0}^{-1}=\nu(\Adj_{\widetilde{w_0}t}(x))^{-1},$$ or equivalently
\begin{equation}\label{Equation: Expression for u_-}u_{-}=\widetilde{w_0}^{-1}\nu(\Adj_{\widetilde{w_0}t}(x))^{-1}\widetilde{w_0}.\end{equation}

Our lemma follows from \eqref{Equation: Expression for u}, \eqref{Equation: Expression for t}, and \eqref{Equation: Expression for u_-}, along with the fact that $\nu$ and $\gamma$ are holomorphic.
\end{proof}

\section{The relationship between $\mathcal{O}_{\text{KT}}$ and $\mathcal{Z}_{\mathfrak{g}}$}\label{Section: The relationship}
We now relate the two integrable systems $\overline{F}:\mathcal{O}_{\text{KT}}\rightarrow\mathbb{C}^r$ and $\widetilde{F}:\mathcal{Z}_{\mathfrak{g}}\rightarrow\mathbb{C}^r$. Section \ref{Subsection: A holomorphic map 2} provides the main ingredients, namely holomorphic maps $\beta:\mathcal{V}\rightarrow\mathcal{S}$ (see \eqref{Equation: Definition of beta}) and $\delta:\mathcal{V}\rightarrow U$ (see \eqref{Equation: Definition of delta}) with certain desirable properties  (Lemmas \ref{Lemma: Uniquely conjugate}, \ref{Lemma: Uniqueness of delta}, and \ref{Lemma: Open stabilizer}). This leads to Section \ref{Subsection: Proof of Theorem 1}, which proves Theorem \ref{Theorem: Main theorem} via three propositions and two lemmas.

\subsection{Some preliminaries}\label{Subsection: A holomorphic map 2}
Recall the isomorphism \eqref{Equation: Kostant variety isomorphism} and let $\pi_U:U\times\mathcal{S}\rightarrow U$ and $\pi_{\mathcal{S}}:U\times\mathcal{S}\rightarrow\mathcal{S}$ denote the canonical projection maps. The holomorphic map 
\begin{equation}\label{Equation: Definition of beta}\beta:\mathcal{V}\rightarrow\mathcal{S},\quad x\mapsto \pi_{\mathcal{S}}(\psi^{-1}(x)),\quad x\in\mathcal{V}\end{equation}
then has the following important property.
\begin{lemma}\label{Lemma: Uniquely conjugate}
If $x\in\mathcal{V}$, then $\beta(x)$ is the unique point in $\mathcal{S}$ with the property of being $G$-conjugate to $x$.
\end{lemma}

\begin{proof}
We have $\psi^{-1}(x)=(u,\beta(x))$ for some $u\in U$, i.e. $\Adj_u(\beta(x))=x$. In particular, $x$ and $\beta(x)$ are $G$-conjugate. Uniqueness follows from the fact that no two distinct elements of $\mathcal{S}$ are $G$-conjugate. 
\end{proof}

Now define another holomorphic map as
\begin{equation}\label{Equation: Definition of delta}\delta:\mathcal{V}\rightarrow U,\quad x\mapsto\bigg(\pi_U(\psi^{-1}(\theta(x)))\bigg)^{-1},\quad x\in\mathcal{V}.
\end{equation}

\begin{lemma}\label{Lemma: Uniqueness of delta}
If $x\in\mathcal{V}$, then $\delta(x)$ is the unique element of $U$ satisfying $\beta(x)=\Adj_{\delta(x)}(\theta(x))$.
\end{lemma}

\begin{proof}
Lemmas \ref{Lemma: Conjugate} and \ref{Lemma: Uniquely conjugate} imply that $\beta(x)$ is the unique element of $\mathcal{S}$ with the property of being $G$-conjugate to $\theta(x)$. It follows that $\psi^{-1}(\theta(x))=(u,\beta(x))$ for some $u\in U$, or equivalently $\beta(x)=\Adj_{u^{-1}}(\theta(x))$. At the same time, \eqref{Equation: Definition of delta} forces $u^{-1}=\delta(x)$ to hold. We conclude that $\beta(x)=\Adj_{\delta(x)}(\theta(x))$.

To establish uniqueness, suppose that $u\in U$ satisfies $\beta(x)=\Adj_{u}(\theta(x))$. We then have $$\Adj_{u^{-1}}(\beta(x))=\theta(x)=\Adj_{\delta(x)^{-1}}(\beta(x)),$$ so that $\psi(u^{-1},\beta(x))=\psi(\delta(x)^{-1},\beta(x))$. Since $\psi$ is an isomorphism, we must have $u=\delta(x)$. This completes the proof.  
\end{proof}

One immediate consequence of Lemma \ref{Lemma: Uniqueness of delta} is the identity \begin{equation}\label{Equation: Identity} G_{\beta(x)}=\delta(x)G_{\theta(x)}\delta(x)^{-1},\quad x\in\mathcal{V}.\end{equation} The following variant of \eqref{Equation: Identity} is needed in the next section. 

\begin{lemma}\label{Lemma: Open stabilizer}
If $x\in\mathcal{V}$, then $G_{\beta(x)}^*=\delta(x)G_{\theta(x)}^*\delta(x)^{-1}$.
\end{lemma}

\begin{proof}
We first establish the inclusion $\delta(x)G_{\theta(x)}^*\delta(x)^{-1}\subseteq G_{\beta(x)}^*$. Suppose that $g\in G_{\theta(x)}^*$ and write $g=\widetilde{w_0}u_{-}tu$ for some $u_{-}\in U_{-}$, $t\in T$, and $u\in U$. We have
$$\delta(x)g\delta(x)^{-1} = \delta(x)\widetilde{w_0}u_{-}tu\delta(x)^{-1}=\widetilde{w_0}(\widetilde{w_0}^{-1}\delta(x)\widetilde{w_0}u_{-})t(u\delta(x)^{-1}),$$ while we observe that $$\widetilde{w_0}^{-1}\delta(x)\widetilde{w_0}u_{-}\in (\widetilde{w_0}^{-1}U\widetilde{w_0})u_{-}=U_{-}\quad\text{and}\quad u\delta(x)^{-1}\in U.$$ It follows that $\delta(x)g\delta(x)^{-1}\in G^*$. At the same time, \eqref{Equation: Identity} implies that $\delta(x)g\delta(x)^{-1}\in G_{\beta(x)}$. We conclude that $\delta(x)g\delta(x)^{-1}\in G_{\beta(x)}^*$, which establishes the inclusion $\delta(x)G_{\theta(x)}^*\delta(x)^{-1}\subseteq G_{\beta(x)}^*$. An analogous argument gives the opposite inclusion.   
\end{proof}

\subsection{Proof of Theorem \ref{Theorem: Main theorem}}\label{Subsection: Proof of Theorem 1}
We are now equipped to prove Theorem \ref{Theorem: Main theorem}, the main result of this paper. Note that Proposition \ref{Proposition: Dense} shows $\mathcal{V}$ to be dense in $\mathcal{O}_{\text{KT}}$, and that Theorem \ref{Theorem: Main theorem}(i) follows immediately from \eqref{Equation: Definition of open Kostant-Toda}. To verify the rest of Theorem \ref{Theorem: Main theorem}, recall the holomorphic maps $\lambda:\mathcal{V}\rightarrow G^*$, $\beta:\mathcal{V}\rightarrow\mathcal{S}$, and $\delta:\mathcal{V}\rightarrow U$ and their properties. We have $\lambda(x)\in G_{\theta(x)}^*\subseteq G_{\theta(x)}$ for all $x\in\mathcal{V}$, which by \eqref{Equation: Identity} implies that $\delta(x)\lambda(x)\delta(x)^{-1}\in G_{\beta(x)}$. It follows that $(\delta(x)\lambda(x)\delta(x)^{-1},\beta(x))\in\mathcal{Z}_{\mathfrak{g}}$, giving rise to the holomorphic map
\begin{equation}\label{Equation: Definition of biholomorphism}\varphi:\mathcal{V}\rightarrow\mathcal{Z}_{\mathfrak{g}},\quad x\mapsto (\delta(x)\lambda(x)\delta(x)^{-1},\beta(x)),\quad x\in\mathcal{V}.\end{equation} This map is compatible with the integrable systems $\overline{F}=(\overline{f}_1,\ldots,\overline{f}_r):\mathcal{O}_{\text{KT}}\rightarrow\mathbb{C}^r$ and $\widetilde{F}=(\widetilde{f}_1,\ldots,\widetilde{f}_r):\mathcal{Z}_{\mathfrak{g}}\rightarrow\mathbb{C}^r$ in the following sense.

\begin{proposition}
We have a commutative diagram
\begin{align*}
\xymatrixrowsep{4pc}\xymatrixcolsep{4pc}\xymatrix{
	\mathcal{V} \ar[rd]_{\overline{F}\big\vert_{\mathcal{V}}} \ar[rr]^{\varphi} & & \mathcal{Z}_{\mathfrak{g}} \ar[ld]^{\widetilde{F}} \\
	& \mathbb{C}^r & 
}.
\end{align*} 
\end{proposition}

\begin{proof}
Our task is to prove that $\widetilde{f}_i\circ\varphi=\overline{f}_i\big\vert_{\mathcal{V}}$ for all $i\in\{1,\ldots,r\}$. Given $x\in\mathcal{V}$ and $i\in\{1,\ldots,r\}$, note that \eqref{Equation: Definition of new f_i} gives
\begin{equation}\label{Equation: Simple}\widetilde{f}_i(\varphi(x))=\widetilde{f}_i(\delta(x)\lambda(x)\delta(x)^{-1},\beta(x))=f_i(\beta(x)).\end{equation}
Now recall that $x$ and $\beta(x)$ are $G$-conjugate (by Lemma \ref{Lemma: Uniquely conjugate}). Since $f_i$ is $G$-invariant, it follows that $f_i(\beta(x))=f_i(x)$. This combines with the calculation \eqref{Equation: Simple} to show that $\widetilde{f}_i\circ\varphi=\overline{f}_i\big\vert_{\mathcal{V}}$, completing the proof.    
\end{proof}

It remains to prove Theorem \ref{Theorem: Main theorem}(iii) and show that $\varphi$ is an open embedding of complex manifolds. To address the latter issue, recall the open dense subset $\mathcal{D}\subseteq\mathbb{C}^r$ from \ref{Subsection: Some technical results}. Let $\mathcal{D}_{\mathcal{S}}\subseteq\mathcal{S}$ denote the preimage of $\mathcal{D}$ under the isomorphism \eqref{Equation: Variety isomorphism}, i.e. $\mathcal{D}_{\mathcal{S}}:=(F\big\vert_{\mathcal{S}})^{-1}(\mathcal{D})$. It follows that $G^*\times\mathcal{D}_{\mathcal{S}}$ is an open subset of $G\times\mathcal{S}$, so that $$\mathcal{W}:=(G^*\times\mathcal{D}_{\mathcal{S}})\cap\mathcal{Z}_{\mathfrak{g}}$$ defines an open subset of $\mathcal{Z}_{\mathfrak{g}}$.

\begin{lemma}\label{Lemma: Image}
The image of $\varphi$ is $\mathcal{W}$.
\end{lemma}

\begin{proof}
Suppose that $x\in\mathcal{V}$, recalling that $x$ and $\beta(x)$ are $G$-conjugate (see Lemma \ref{Lemma: Uniquely conjugate}). This implies that $F(\beta(x))=F(x)$, while we note that \eqref{Equation: Definition of open Kostant-Toda} gives $F(x)\in\mathcal{D}$. It follows that $F(\beta(x))\in\mathcal{D}$, or equivalently $\beta(x)\in (F\big\vert_{\mathcal{S}})^{-1}(\mathcal{D})=\mathcal{D}_{\mathcal{S}}$. At the same time, Lemma \ref{Lemma: Open stabilizer} tells us that $\delta(x)\lambda(x)\delta(x)^{-1}\in G_{\beta(x)}^*$. We conclude that $\varphi(x)\in G^*\times\mathcal{D}_{\mathcal{S}}$, i.e. $\varphi(x)\in\mathcal{W}$, implying that $\varphi(\mathcal{V})\subseteq\mathcal{W}$. 

To prove that $\mathcal{W}\subseteq \varphi(\mathcal{V})$, let $(g,x)\in\mathcal{W}$ be given. The restriction of $F$ to $\mathfrak{t}$ is surjective (see the proof of \cite[Proposition 10]{KostantLie}), allowing us to find $y\in\mathfrak{t}$ with $F(y)=F(x)$. Set $z:=\xi+y\in\mathcal{O}_{\text{KT}}$ and note that $F(z)=F(x)$ (see \cite[Corollary 3.1.43]{Chriss}). Note also that $x\in\mathcal{S}\subseteq\mathfrak{g}_{\text{reg}}$, while \cite[Lemma 10]{KostantLie} implies that $z\in\mathfrak{g}_{\text{reg}}$. An application of \cite[Theorem 2]{KostantLie} then shows $z$ and $x$ to be $G$-conjugate. At the same time, the definition of $\mathcal{W}$ implies that $F(z)=F(x)\in\mathcal{D}$, so that $z\in\overline{F}^{-1}(\mathcal{D})=\mathcal{V}$. These last sentences combine with Lemma \ref{Lemma: Uniquely conjugate} to imply that $\beta(z)=x$. Lemma \ref{Lemma: Open stabilizer} then tells us that $G_x^*=\delta(z)G_{\theta(z)}^*\delta(z)^{-1}$. Now observe that $g\in G_x^*$, so that we may consider the element $$h:=\delta(z)^{-1}g\delta(z)\in G_{\theta(z)}^*$$ and its image $$v:=\tau_{\theta(z)}(h)\in \overline{F}^{-1}(\overline{F}(z))$$ under Kostant's isomorphism $\tau_{\theta(z)}$ (see Theorem \ref{Theorem: Kostant fibre isomorphism}). Note that $v\in\mathcal{V}$, as $F(v)=F(z)=F(x)\in\mathcal{D}$. Note also that $$\lambda(v)=\tau_{\theta(z)}^{-1}(v)=\tau_{\theta(z)}^{-1}(\tau_{\theta(z)}(h))=h,$$ where the first equality comes from Lemma \ref{Lemma: Property of lambda}.

We claim that $\varphi(v)=(g,x)$. To this end, Lemma \ref{Lemma: Precisely conjugate} and the condition $v\in \overline{F}^{-1}(\overline{F}(z))$ imply that $v$ and $z$ are $G$-conjugate. It then follows from Lemmas \ref{Lemma: Conjugate} and \ref{Lemma: Uniquely conjugate} that $\theta(v)=\theta(z)$ and $\beta(v)=\beta(z)$, respectively. By Lemma \ref{Lemma: Uniqueness of delta}, we must have $\delta(v)=\delta(z)$. We now note that
\begin{align*}
\varphi(v) & = (\delta(v)\lambda(v)\delta(v)^{-1},\beta(v))\\
& = (\delta(z)\lambda(v)\delta(z)^{-1},\beta(z))\\
& = (\delta(z)\lambda(v)\delta(z)^{-1},x)\\
& = (\delta(z)h\delta(z)^{-1},x)\\
& = (\delta(z)\delta(z)^{-1}g\delta(z)\delta(z)^{-1},x)\\
& = (g,x).
\end{align*}
This establishes that $\mathcal{W}\subseteq \varphi(\mathcal{V})$, completing the proof. 
\end{proof}

\begin{proposition}
The map $\varphi$ is an open embedding of complex manifolds.
\end{proposition}

\begin{proof}
As noted just before Lemma \ref{Lemma: Image}, $\mathcal{W}$ is an open subset of $\mathcal{Z}_{\mathfrak{g}}$. Lemma \ref{Lemma: Image} therefore reduces our task to one of proving that $\varphi$ is injective. To this end, suppose that $x,y\in\mathcal{V}$ have the same images under $\varphi$, i.e.
$$\beta(x)=\beta(y)\quad\text{and}\quad \delta(x)\lambda(x)\delta(x)^{-1}=\delta(y)\lambda(y)\delta(y)^{-1}.$$
By applying Lemma \ref{Lemma: Uniquely conjugate} to the first equation, we conclude that $x$ and $y$ are $G$-conjugate. Lemma \ref{Lemma: Conjugate} then gives $\theta(x)=\theta(y)$, which by Lemma \ref{Lemma: Uniqueness of delta} implies that $\delta(x)=\delta(y)$. Since $\delta(x)\lambda(x)\delta(x)^{-1}=\delta(y)\lambda(y)\delta(y)^{-1}$, we conclude that $\lambda(x)=\lambda(y)$. The definition of $\lambda$ (see \eqref{Equation: Definition of lambda}) and the fact that $\theta(x)=\theta(y)$ combine to imply that $\tau_{\theta(x)}^{-1}(x)=\tau_{\theta(x)}^{-1}(y)$. This establishes that $x=y$.
\end{proof}

It remains only to prove Theorem \ref{Theorem: Main theorem}(iii). To this end, recall our discussion of the vectors $(d_xf_i)^{\vee}\in\mathfrak{g}_x$ for $x\in\mathfrak{g}$ and $i\in\{1,\ldots,r\}$ (see \ref{Subsection: The integrable system}).

\begin{lemma}\label{Lemma: Differential identity}
If $g\in G$, $x\in\mathfrak{g}$, and $y:=\Adj_g(x)$, then $\Adj_g((d_xf_i)^{\vee}) = (d_yf_i)^{\vee}$ for all $i\in\{1,\ldots,r\}$.
\end{lemma}

\begin{proof}
Since $f_i$ is a $G$-invariant polynomial on $\mathfrak{g}$, we have $f_i\circ\Adj_{g^{-1}}=f_i$. Differentiating both sides at $y$ yields $d_xf_i\circ\Adj_{g^{-1}}=d_yf_i$. It follows that for all $z\in\mathfrak{g}$,
\begin{align*}
\langle\Adj_g((d_xf_i)^{\vee}),z\rangle & = \langle (d_xf_i)^{\vee},\Adj_{g^{-1}}(z)\rangle\\
& = (d_xf_i\circ\Adj_{g^{-1}})(z)\\
& = d_yf_i(z)\\
& = \langle (d_yf_i)^{\vee},z\rangle.
\end{align*}

We conclude that $\Adj_g((d_xf_i)^{\vee}) = (d_yf_i)^{\vee}$.
\end{proof}

Now observe that Theorem \ref{Theorem: Main theorem}(iii) amounts to the following proposition, in which $H(\overline{f}_i)$ (resp. $H(\widetilde{f}_i)$) denotes the Hamiltonian vector field of $\overline{f}_i$ on $\mathcal{O}_{\text{KT}}$ (resp. $\widetilde{f}_i$ on $\mathcal{Z}_{\mathfrak{g}}$). 

\begin{proposition}
We have $H(\widetilde{f}_i)_{\varphi(x)}=d_x\varphi(H(\overline{f}_i)_x)$ for all $i\in\{1,\ldots,r\}$ and $x\in\mathcal{V}$.
\end{proposition}

\begin{proof}
Since $\overline{F}=(\overline{f}_1,\ldots,\overline{f}_r):\mathcal{O}_{\text{KT}}\rightarrow\mathbb{C}^r$ is a completely integrable system, $H(\overline{f}_i)$ is tangent to the level sets of $\overline{F}$ for each $i\in\{1,\ldots,r\}$. An application of Lemma \ref{Lemma: Precisely conjugate} then reveals the following: if two points in $\mathcal{O}_{\text{KT}}$ lie on the same integral curve of $H(\overline{f}_i)$, then these points are $G$-conjugate to one another. By Lemmas \ref{Lemma: Conjugate} and \ref{Lemma: Uniquely conjugate}, $\theta$ and $\beta$ must be constant-valued along each integral curve of $H(\overline{f}_i)$ in $\mathcal{V}$. Lemma \ref{Lemma: Uniqueness of delta} then implies that $\delta$ is also constant-valued along each integral curve of $H(\overline{f}_i)$ in $\mathcal{V}$. Together with the definition \eqref{Equation: Definition of biholomorphism} of $\varphi$, the previous two sentences entail the following identity for all $x\in\mathcal{V}$:
\begin{equation}\label{Equation: Differential calculation} d_x\varphi(H(\overline{f}_i)_x)=\bigg(d_{\lambda(x)}C_{\delta(x)}\bigg(d_x\lambda(H(\overline{f}_i)_x)\bigg),0\bigg)\in T_{\varphi(x)}\mathcal{Z}_{\mathfrak{g}},\end{equation}
where $C_{\delta(x)}:G\rightarrow G$ is conjugation by $\delta(x)$. Note that $$T_{\varphi(x)}\mathcal{Z}_{\mathfrak{g}}\subseteq T_{(h,\beta(x))}(G\times\mathfrak{g})=T_hG\oplus\mathfrak{g}$$ with $h:=\delta(x)\lambda(x)\delta(x)^{-1}$, and that the right hand side of \eqref{Equation: Differential calculation} should be interpreted as an element of $T_hG\oplus\mathfrak{g}$. 

Now set $g:=\lambda(x)$ in \eqref{Equation: Relation between vector fields}, noting that $\tau_{\theta(x)}(g)=x$ must also hold (see \eqref{Equation: Definition of lambda}. We obtain $$d_{\lambda(x)}\tau_{\theta(x)}(d_eL_{\lambda(x)}((d_{\theta(x)}f_i)^{\vee}))=H(\overline{f}_i)_x,$$ or equivalently
\begin{equation}\label{Equation: Equivalent statement}(d_{\lambda(x)}\tau_{\theta(x)})^{-1}(H(\overline{f}_i)_x)=d_eL_{\lambda(x)}(d_{\theta(x)}f_i)^{\vee}).\end{equation} At the same time, Lemma  \ref{Lemma: Property of lambda} and $H(\overline{f}_i)$ being tangent to $\overline{F}^{-1}(\overline{F}(x))$ imply that $\lambda$ and $\tau_{\theta(x)}^{-1}$ agree on any integral curve through of $H(\overline{f}_i)$ through $x$. We conclude that 
$$d_x\lambda(H(\overline{f}_i)_x)=d_x(\tau_{\theta(x)}^{-1})(H(\overline{f}_i)_x)=(d_{\lambda(x)}\tau_{\theta(x)})^{-1}(H(\overline{f}_i)_x),$$
so that \eqref{Equation: Equivalent statement} may be written as
$$d_x\lambda(H(\overline{f}_i)_x)=d_eL_{\lambda(x)}((d_{\theta(x)}f_i)^{\vee}).$$ After substituting the right hand side into \eqref{Equation: Differential calculation}, we find that
\begin{equation}\label{Equation: Second differential calculation} d_x\varphi(H(\overline{f}_i)_x)=\bigg(d_{\lambda(x)}C_{\delta(x)}\bigg(d_eL_{\lambda(x)}((d_{\theta(x)}f_i)^{\vee})\bigg),0\bigg).\end{equation} 

We next observe that $h:=\delta(x)\lambda(x)\delta(x)^{-1}$ satisfies $L_h\circ C_{\delta(x)}=C_{\delta(x)}\circ L_{\lambda(x)}$, which becomes
$$d_eL_h\circ\Adj_{\delta(x)}=d_{\lambda(x)}C_{\delta(x)}\circ d_eL_{\lambda(x)}$$ upon the differentiation of both sides at $e\in G$. Modifying \eqref{Equation: Second differential calculation} accordingly, we obtain
$$d_x\varphi(H(\overline{f}_i)_x)=\bigg(d_eL_h\bigg(\Adj_{\delta(x)}((d_{\theta(x)}f_i)^{\vee})\bigg),0\bigg).$$ Now recall that $\Adj_{\delta(x)}(\theta(x))=\beta(x)$ (see Lemma \ref{Lemma: Uniqueness of delta}), so that Lemma \ref{Lemma: Differential identity} yields $\Adj_{\delta(x)}((d_{\theta(x)}f_i)^{\vee})=(d_{\beta(x)}f_i)^{\vee}$. It follows that $$d_x\varphi(H(\overline{f}_i)_x)=\bigg(d_eL_h((d_{\beta(x)}f_i)^{\vee})),0\bigg).$$ By Proposition \ref{Proposition: Description of Hamiltonian vector fields}, the right hand side is precisely the result of evaluating $H(\widetilde{f}_i)$ at $\varphi(x)=(h,\beta(x))$. This completes the proof. 
\end{proof}

\bibliographystyle{acm} 
\bibliography{Centralizer}
\end{document}